\documentclass[11pt,a4paper]{amsart}
\usepackage{amsmath,amssymb,amsthm, amsfonts,enumerate,mathtools,pinlabel,float,makecell}
\usepackage[mathscr]{eucal}
\usepackage{amsbsy}
\usepackage{graphicx}
\usepackage{multirow}
\usepackage{tabularx}
\usepackage{lscape}
\usepackage[myheadings]{fullpage}
\usepackage{rotating}
\usepackage{tikz-cd,tikz,pgf}
\usepackage{tikz-cd}
\usetikzlibrary{arrows.meta, bending}
\tikzset{C/.style={circle, minimum size=8mm,
                   node contents={},
                   append after command={\pgfextra{%
        \draw[-{Straight Barb[flex']}](\tikzlastnode.150) arc (150:450:4mm);}
                }}
        }
\usepackage{caption}
\usepackage{subcaption}
\usepackage{cite}
\usepackage{array}
\usepackage{xcolor}
\usepackage{calc}
\usepackage{comment}
\usepackage{soul}
\usepackage{hyperref}
\hypersetup{
    colorlinks=true,  
    linkcolor=blue,
    citecolor=blue,
}

\newtheorem{cor*}{Corollary}
\newtheorem{thm*}{Theorem}
\newtheorem{lem*}{Lemma}
\newtheorem{prop*}{Proposition}
\newtheorem{qstn*}{Question}
\newtheorem{theorem}{Theorem}[section]
\newtheorem{cor}{Corollary}[theorem]

\newtheorem{algo}[theorem]{Algorithm}

\newtheorem{algo*}{Algorithm}
\theoremstyle{definition}
\newtheorem{defn}[theorem]{Definition}

\newtheorem{rem}[theorem]{Remark}

\newcommand{\Mod}{\mathrm{Mod}}
\newcommand{\Teich}{\mathrm{Teich}}
\newcommand{\T}{Teichm\"uller }

\newcommand{\Diffeo}{\mathrm{Diffeo}^{+}}

\newcommand{\Diff}{\mathrm{Diff}}
\newcommand{\HypMet}{\mathrm{HypMet}}
\newcommand{\DF}{\mathrm{DF}}
\newcommand{\PSL}{\mathrm{PSL}}
\newcommand{\PGL}{\mathrm{PGL}}

\usepackage{booktabs}

\everymath{\displaystyle}

\makeatletter
\@namedef{subjclassname@2022}{\textup{2022} Mathematics Subject Classification}
\makeatother

\begin{document}

\title[Teichm\"{u}ller isometries induced by certain irreducible periodic mapping classes]{Teichm\"{u}ller isometries induced by certain \\ irreducible periodic mapping classes}

\author{Atreyee Bhattacharya}
\address{(A. Bhattacharya) Department of Mathematics\\
Indian Institute of Science Education and Research Bhopal\\
Bhopal Bypass Road, Bhauri \\
Bhopal 462 066, Madhya Pradesh\\
India}
\email{atreyee@iiserb.ac.in}
\urladdr{https://sites.google.com/iiserb.ac.in/homepage-atreyee-bhattacharya/home?authuser=1}

\author{Satyajit Maity}
\address{(S. Maity) Department of Mathematics\\
Indian Institute of Science Education and Research Bhopal\\
Bhopal Bypass Road, Bhauri \\
Bhopal 462 066, Madhya Pradesh\\
India}
\email{msatyajit.194@gmail.com}

\author{Kashyap Rajeevsarathy}
\address{(K. Rajeevsarathy) Department of Mathematics\\
Indian Institute of Science Education and Research Bhopal\\
Bhopal Bypass Road, Bhauri \\
Bhopal 462 066, Madhya Pradesh\\
India}
\email{kashyap@iiserb.ac.in}
\urladdr{https://home.iiserb.ac.in/$_{\widetilde{\phantom{n}}}$kashyap/}



\begin{abstract} Let $\mathrm{Mod}(S_g)$ be the mapping class group of the closed orientable surface of genus $g \geq 2$, and let $\mathrm{Teich}(S_g)$ be the Teichm\"{u}ller space of $S_g$. In this paper, we provide an algorithm for determining the isometries induced in $\mathrm{Teich}(S_g)$ by certain irreducible periodic mapping classes in Fenchel-Nielsen coordinates. As a demonstration of this method, we provide a description of the isometry induced by periodic mapping class of order $4g+2$ and also derive some of its geometric properties. 

\end{abstract}

\maketitle

\section{Introduction}\label{sec:intro}
Let \(\Mod(S_g)\) represent the mapping class group of the closed orientable surface \(S_g\) of genus \(g \geq 2\), and let \(\Teich(S_g)\) denote the corresponding Teichm\"{u}ller space. The Nielsen Realization Theorem \cite{Kerckhoff, Neilsen1932} states that any finite subgroup of \(\Mod(S_g)\) can be realized as a group of isometries for some hyperbolic metric on \(S_g\). Recently, \cite{Realization1} introduced methods for constructing explicit hyperbolic metrics on \(\Mod(S_g)\) that realize finite cyclic subgroups of \(\Mod(S_g)\). This was followed by \cite{Realization2}, which provided explicit parametrizations of the fixed point sets of these cyclic actions as totally geodesic Kähler submanifolds of \(\Teich(S_g)\), with respect to the Weil-Petersson metric. It is well-established that \(\Mod(S_g)\) acts properly discontinuously on the Teichm\"{u}ller space by isometries with respect to both the Teichm\"{u}ller and Weil-Petersson metrics. However, the explicit descriptions of these isometric actions in terms of the Fenchel-Nielsen coordinates of the Teichmüller space are not completely understood. In particular, aside from surface rotations, the Fenchel-Nielsen coordinates of the fixed points of the isometric action of finite cyclic subgroups \(H = \langle F \rangle\) of \(\Mod(S_g)\) on the Teichmüller space have not been thoroughly investigated.
 
The main focus of this paper is to examine the isometry induced on the Teichmüller space by an irreducible periodic mapping class \( F \in \operatorname{Mod}(S_g) \). This mapping class has a Nielsen representative (a homeomorphism of the same order) that possesses at least one fixed point on \( S_g \). These mapping classes, known as \textit{irreducible Type 1 actions}, are particularly significant as they serve as the building blocks for all periodic mapping classes of \( S_g \), as demonstrated in works such as \cite{Realization1, Realization2}. Furthermore, it has been established that for an irreducible Type 1 action \( F \), the induced map \( F_{\#} \) on \( \operatorname{Teich}(S_g) \) has a unique fixed point. This fixed point can be represented by a canonical semi-regular hyperbolic polygon \( \mathcal{P}_F \) with an appropriate side pairing that realizes \( F \) as a rotation. Utilizing the geometric realizations detailed in \cite{Realization1, Realization2}, an algorithmic description of the fixed point sets of isometric actions on \( \operatorname{Teich}(S_g) \) induced by specific periodic mapping classes has been previously provided in \cite{ASK}. In this paper, we present an algorithm to determine the induced isometry \( F_{\#} \) in Fenchel-Nielsen coordinates.
\begin{algo*}\label{algo:1}
Let $F\in \Mod(S_g)$, and $\xi, \tilde{\xi}$ be hyperbolic metrics on $S_g$ such that under the induced action $F_{\#}: \Teich(S_g) \longrightarrow \Teich(S_g)$, $F_{\#}([\xi])=[\tilde{\xi}] \in \Teich(S_g)$. 
\begin{enumerate}[\textit{Step} 1.] 
\item If $\mathcal{P}=\{\gamma_i\}_{i\in I}$ is the chosen pants decomposition for $S_g$, find the corresponding pants decomposition determined by $T$, that is $\tilde{\mathcal{P}}=\{T(\gamma_i)\}_{i\in I}$.
\item To compute the length of each $\tilde{\gamma_i}:=T(\gamma_i)$, cut the hyperbolic surface $(S_g, \xi)$ appropriately (possibly along the curve $\gamma_i$ if $[\gamma_i] \neq [\tilde{\gamma_i}]$) and lift it to the universal cover (the upper-half plane $\mathbb{H}^2$ or Poincar\'e disc $\mathbb{D}$).
\item In the universal cover, to find the corresponding unique geodesic representative $[\tilde{\gamma_i}]$, parametrize the family of simple closed smooth curves homotopic to $\tilde{\gamma_i}$, and minimize the length over the family of such curves.
\item For determining the twist parameters $\tilde{t_i}$ of $\tilde{\gamma_i}$, cut the surface $(S_g, \tilde{\xi})$ suitably along the curves $\tilde{\gamma_i}$ and lift it to the universal cover, and then use the method as described in Section \ref{Sec:Method1}. 
\end{enumerate}
\end{algo*}
\noindent  A natural simple-minded approach to calculating \( F_{\#} \) involves first factorizing a given periodic mapping class into a product of Humphries' twists~\cite{SH}, and then determining \( F_{\#} \) by composing the Teichm\"{u}ller isometries induced by these Dehn twist factors. We briefly discuss this method in Section~\ref{Sec:Method1} (see Theorems \ref{thm.4.1} - \ref{thm: Induced Tc2} for more details). However, this method has two major drawbacks: an explicit factorization of a mapping class may not be available, and even when it is, computing the composition of the induced actions from multiple Dehn twists becomes increasingly complex as the surface genus increases. It is worth mentioning that Algorithm~\ref{algo:1} effectively addresses both of these challenges.

We further demonstrate the efficiency of Algorithm~\ref{algo:1} by applying it to analyze the induced isometry of the irreducible Type 1 action \( F_g \) of order \( 4g + 2 \) on \( S_g \). This periodic mapping class is significant because it has the highest possible order in $\mathrm{Mod}(S_g)$~\cite{AW}. Additionally, \( F_g \) is a chain map, meaning it can be expressed as the product of Dehn twists around isotopy classes of \( 2g \) essential simple closed curves that form a chain on \( S_g \) (see Definition \ref{chain} for details). As an application of Algorithm \ref{algo:1}, we derive a closed-form expression for the unique fixed point of the induced isometric action \( F_{\#} \) on the \T space, in terms of the Fenchel-Nielsen coordinates.
\begin{thm*}\label{thm:fixedpt_gen_intro}
For $g\geq 2$, the unique fixed point $m_g$ of $F_{\#}$ is given in terms of the Fenchel-Nielsen coordinates, as $m_g = (\gamma_1, \gamma_2, \dots, \gamma_{3g-3}, t_1, t_2, \dots, t_{3g-3}),$
where the twist parameters $t_i = 0$ for all $1 \le i \le 3g-3$, and the length parameters $\gamma_i$ are defined by the palindromic sequence
\begin{equation*}
    \gamma_{3j} = \gamma_{3j+1} = \gamma_{3g-3-3j} = \gamma_{3g-2-3j} =2\operatorname{arcosh}\left(Y_j(g)\right) \quad \text{for } 1 \le j \le \left\lfloor \frac{g-1}{2} \right\rfloor,
\end{equation*}
and all the remaining length coordinates $\gamma_i$ are given by $\gamma_i = 2\operatorname{arcosh}\left(X(g)\right)$, where $X, Y_j$ are functions of $g$ as follows
\begin{align*}
    X(g) &= 1 + 2\cos\left(\frac{\pi}{g+1}\right), \\
    Y_j(g) &= \frac{\cos\left(\frac{\pi}{g+1}\right) - \cos\left(\frac{2j+2}{g+1}\pi\right)}{1 - \cos\left(\frac{\pi}{g+1}\right)}.
\end{align*}
\end{thm*}
In particular, for $g=2$, we also obtain the description of the induced action $F_{\#}$ as follows. 
\begin{thm*} \label{thm:order10-intro}
Let $\mathcal{P}_2=\{\gamma_1, \gamma_2, \gamma_3\}$ denote a pants decomposition of $S_2$ and $\Gamma_2=\{\gamma_1, c_1, \gamma_2, c_2 \}$ be a filling chain collection where $\mathcal{C}_2=\{c_1, c_2, c_3\}$ denotes the corresponding collection of seams curves as shown in Figure \ref{Fig. Pants_seams_curves_S2}. Then the induced action $F_{\#}: \Teich(S_2) \longrightarrow \Teich(S_2)$ of the mapping class $F=T_{\gamma_1}T_{c_1}T_{\gamma_2}T_{c_2}$ is given by 
$$(\gamma_1, \gamma_2, \gamma_3, t_1, t_2, t_3) \mapsto (c_1, c_2, c_3, \Tilde{t}_1, \Tilde{t}_2, \Tilde{t}_3)$$
 where $c_1, c_2, c_3$ are as follows
$$\cosh{\frac{c_i}{2}}=|\cosh(\ln{\sqrt{r_i}})\cosh{\frac{t_i+b_i}{2}} - \frac{a_i}{2\sqrt{r_i}}\sinh{\frac{t_i+b_i}{2}}|$$
and $\Tilde{t}_1, \Tilde{t}_2, \Tilde{t}_3$ are given by an algorithmic method. In particular, the unique fixed point of $F_{\#}$ is of the form
 $m_0:=(x_0, x_0, x_0, 0, 0, 0)\in \mathcal{H}_0$ where $x_0=\operatorname{arcosh} 2=\ln(7+2\sqrt{3})$.
\begin{figure}[h]
     \centering
        \includegraphics[width=0.5\textwidth]{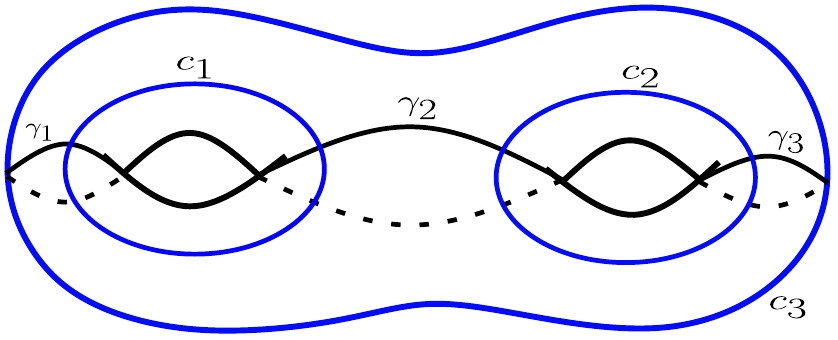}
         \caption{Pants and seams curves on $S_2$.}
        \label{Fig. Pants_seams_curves_S2}
\end{figure}
\end{thm*}  

As another application of Algorithm \ref{algo:1}, we explore some properties of \( F_{\#} \). Considering that $\mathrm{Teich}(S_g)$ with the Weil-Petersson metric is endowed with the standard symplectic structure, we show that there are isotropic submanifolds that are left invariant under the $\langle F_{\#} \rangle$-action (see Theorem \ref{thm. 4g+2 order action on Sg} for further details). Notably, in the special case where \( g=2 \), a specific isotropic submanifold we obtain is Lagrangian. For the sake of brevity, we will state only the result relevant to \( g=2 \) here.
\begin{thm*}
\label{thm3}
Given a pants decomposition $\mathcal{P}_2$ of $S_2$, a filling chain collection $\Gamma_2$, and  corresponding collection $\mathcal{C}_2$ of seams curves as in Theorem \ref{thm:order10-intro}, $F_{\#}$ has the following properties:
\begin{enumerate}[(i)]
\item $\Tilde{t}_1=\Tilde{t}_2=\Tilde{t}_3=0$ and 
$\cosh{\frac{c_i}{2}}=\frac{\cosh{\frac{\gamma_i}{2}}\cosh{\frac{\gamma_{i+1}}{2}} \hspace{0.1cm}+ \hspace{0.1cm}\cosh{\frac{\gamma_{i+2}}{2}}}{\sinh{\frac{\gamma_i}{2}}\sinh{\frac{\gamma_{i+1}}{2}}}$ for $i=1, 2, 3$, with $i+1, i+2$ counted up to congruence modulo $3$.
\item The submanifold $\mathcal{H}:=\{(x_1, x_2, x_3, 0, 0, 0) | x_1, x_2, x_3>0 \} \subset \Teich(S_2)$ is an $F_{\#}$-invariant Lagrangian submanifold of $\Teich(S_2)$ with respect to its natural symplectic structure. 
\item The $1$-dimensional submanifold $\mathcal{H}_0:=\{(x, x, x, 0, 0, 0)| x>0\} \subset \mathcal{H}$ is isotropic and $F_{\#}$-invariant where $F_{\#}(x, x, x, 0, 0, 0)=(f(x), f(x), f(x), 0, 0, 0)$ and $f: \mathbb{R} \to \mathbb{R}_{+}$ is given by $f(x)=2\operatorname{arcosh}(\frac{1}{1-\operatorname{sech}{\frac{x}{2}}})$. 
\item $F_{\#}$ is not an Euclidean isometry.
\end{enumerate} 
\end{thm*}

We conclude the paper with the following consequence of Theorem \ref{thm3}, which describes the asymptotic behavior of the unique fixed point \( m_g \). In particular, this implies that as \( g \to \infty \), the injectivity radius remains bounded from below, thereby ruling out the possibility of a collapse.
\begin{cor*}\label{thm:asymptotic-limit}
 Consider $F_g \in \Mod(S_g)$, its induced action $F_{\#}: \Teich(S_g) \longrightarrow \Teich(S_g)$, and the unique fixed point $m_g \in \Teich(S_g)$ of $F_{\#}$ as mentioned in Theorem \ref{thm. 4g+2  order action on Sg}. As $g \to \infty$, the Fenchel-Nielsen length coordinates $X(g)$ and $Y_j(g)$ of $m_g$ converge point-wise to universal integer values that are independent of $g$ as follows:
\begin{align*}
    \lim_{g \to \infty} X(g) &= 3, \text{ and} \\
    \lim_{g \to \infty} Y_j(g) &= (2j+1)(2j+3) = 4j^2 + 8j + 3.
\end{align*}
\end{cor*}

\section{Preliminaries}
In this section we briefly recall certain basic notions and relevant results pertinent to the central theme of this paper.

\subsection{Humphries generators of the mapping class group}
Let $S_g$ denote a closed, connected, oriented smooth surface of genus $g\geq 0$, $\Diffeo(S_g)$ be the group of its orientation-preserving diffeomorphisms, and $\Mod(S_g)$ be the mapping class group of $S_g$. Given a simple closed curve on $S_g$, let $T_c$ denote the corresponding (left-handed) Dehn twist map about $c$. It is well known \cite{SH} that for $g\geq 0$, $\Mod(S_g)$ is generated by $2g+1$ Dehn twists about  non-separating simple closed curves in $S_g$. The following result provides a precise description of these generators.
\begin{theorem}[{Humphries}\cite{SH}]
\label{thm: Humphries}
For $g \geq 2$, $\Mod(S_g)$ is generated by Dehn twists about the $2g+1$ isotopy classes of non-separating simple closed curves $c_1,\ldots , c_g,\gamma_1,\ldots, \gamma_{g+1}$ in $S_g$ as shown in Figure \ref{fig. Humphries curves}.
\end{theorem}

\begin{figure}[h]
    \centering
    \includegraphics[width=0.8\textwidth]{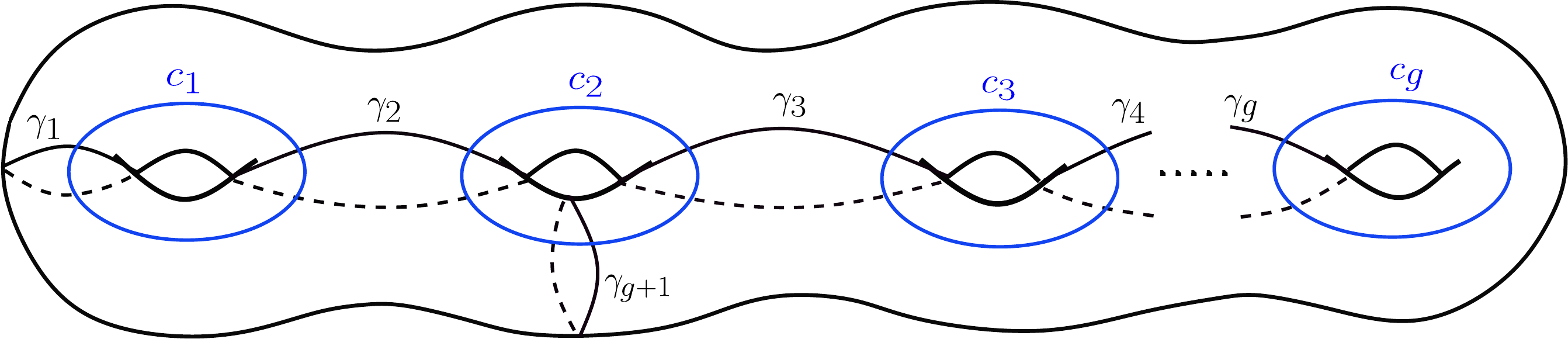}
    \caption{The curves determining the Humphries twists in $\mathrm{Mod}(S_g)$.}
    \label{fig. Humphries curves}
\end{figure}

\subsection{Action of $\Mod(S_g)$ on the \T space of $S_g$} Recall that the \T  space of $S_g$ for $g\geq 2$ is defined to be the quotient space $\Teich(S_g)=\HypMet(S_g)/\Diff_0(S_g)$ where $\HypMet(S_g)$ is the set of all hyperbolic metrics on $S_g$, $\Diff_0(S_g)$ be the group of all diffeomorphisms of $S_g$ isotopic to the identity map, and $\Diff_0(S_g)$ acts on $\HypMet(S_g)$ by pullback of metrics.  
Equivalently, $\Teich(S_g) = \DF(\pi_1(S_g), \PSL_2(\mathbb{R}))/\PGL_2(\mathbb{R})$ which naturally provides the Euclidean topology on $\Teich(S_g)$ where the action of $\PGL_2(\mathbb{R})$ on the set of all discrete faithful representations of $\pi_1(S_g)$ into $\PSL_2(\mathbb{R})$, denoted by $\DF(\pi_1(S_g), \PSL_2(\mathbb{R}))$, is given by conjugacy $(A\cdot\rho)(\alpha)=A\rho(\alpha)A^{-1}.$ 

The action of $\Diff_0(S_g)$ on $\HypMet(S_g)$ induces a natural action of $\Mod(S_g)$ on $\Teich(S_g)$ as follows: Given $F=[f] \in \operatorname{Mod}\left(S_g\right)$ and $[\xi] \in \operatorname{Teich}\left(S_g\right)$, we have $F \cdot[\xi]=\left[f^*(\xi)\right].$ Moreover, it turns out that with respect to the Weil-Petersson metric $\xi_{WP}$, the orientation-preserving group of isometries of the $S_g$ is precisely $\Mod(S_g)$ \cite{MasurWolf2002}.

\subsection{The length and twist Parameters.}\label{sub-sec:twist}
 Let $\gamma$ be an essential simple closed curve on $S_g$ ($g \geq 2$) and $[\xi] \in \Teich(S_g)$, where $\xi \in \operatorname{HypMet}\left(S_g\right)$. If $[\gamma]$ denotes the unique geodesic homotopic to $\gamma$ in $\left(S_g, \xi\right)$, then the $[\xi]$ length of $\gamma$, denoted by $\ell_{\gamma}([\xi])$, is defined as $\ell_{\gamma}([\xi])=\ell_{[\gamma]}(\xi)$ i.e. the $\xi$ length of $[\gamma]$.  
 
Let $\mathcal{P}=\left\{\gamma_1, \gamma_2, \ldots, \gamma_{3 g-3}\right\}$ be a pants decomposition of $S_g$. The length parameters of $\mathcal{P}$ with respect to $[\xi]$ is the $(3g-3)$-tuple $(\ell_{\gamma_1}([\xi]), \ldots , \ell_{\gamma_{3g-3}}([\xi]))$ of lengths of the pants curves with respect to $[\xi]$ as defined above. Let $\mathcal{S}=\left\{c_1, c_2, \ldots, c_{g+1}\right\}$ be a multicurve comprising $g+1$ disjoint simple closed curves on $S_g$, referred to as \textit{seams curves} for the pants decomposition $\mathcal{P}$ (see \cite[Chapter 10]{primer} for details). Then the \textit{twist parameter of $\gamma_i$} with respect to $[\xi]$ is the signed distance (measuring along the unique $\xi$-geodesic representative $\gamma_i^{\prime}$ of $[\gamma_i]$) between the two points where the two common perpendiculars intersect $\gamma_i^{\prime}$ by following any one of the two seams curve passing through it (see \cite{primer, Bruno} for more details). We call a twist to be positive if one of the two common perpendiculars is on the left of the other (in the universal cover), and negative otherwise. Figure \ref{Fig. 3} shows the cases of zero and positive twists, where the twist is taken to be positive if one of the endpoints of the intersection point of $\alpha$ and $\gamma$ is on the left of the other one when lifting the surface into the universal cover.

    \begin{figure}[ht]
     \centering
     \begin{subfigure}[t]{0.2\textwidth}
         \centering
         \includegraphics[width=\textwidth]{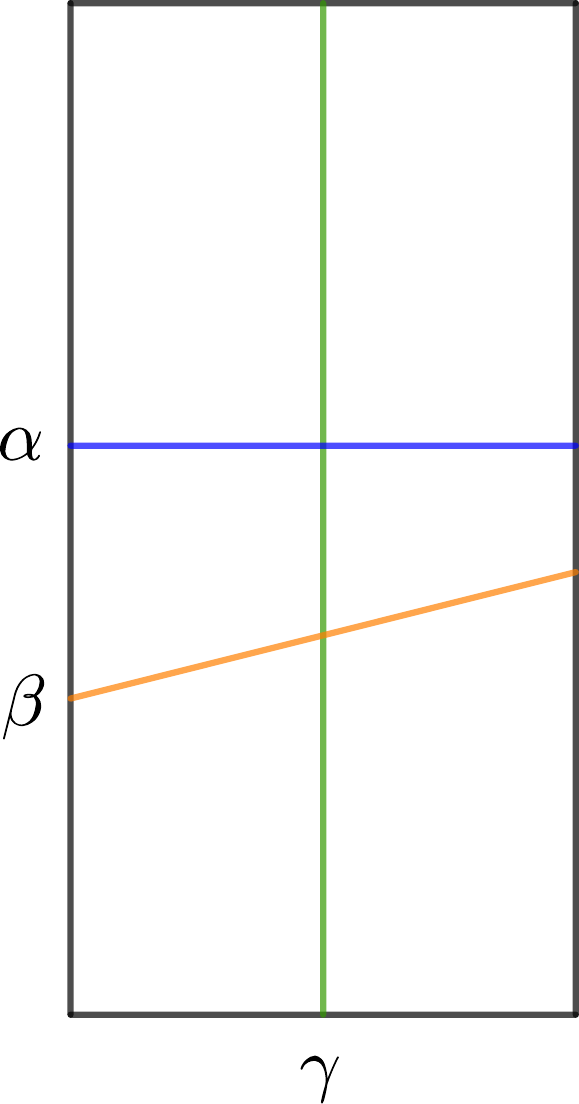}
         \caption{$t_\gamma=0$.}
     \end{subfigure}
     \hspace{2 cm}
     \begin{subfigure}[t]{0.2\textwidth}
         \centering
         \includegraphics[width=\textwidth]{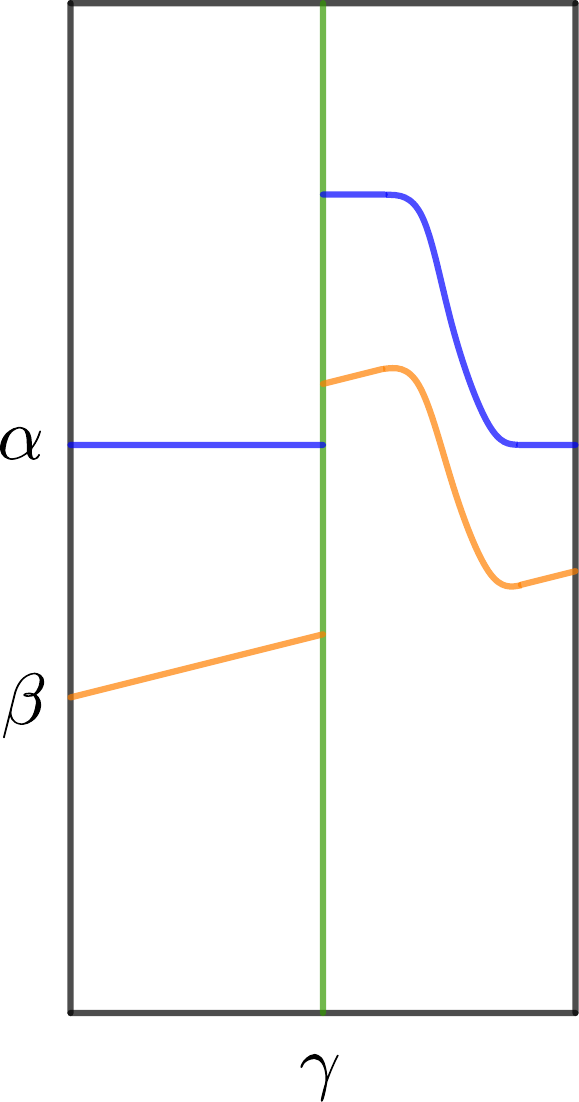}
         \caption{$t_\gamma > 0$.}
     \end{subfigure}
      \caption{Zero twist versus positive twist.}
      \label{Fig. 3}
\end{figure} 

The following theorem due to Fenchel and Nielsen \cite{Fenchel-Nielsen} provides a parametrization of $\Teich(S_g)$ in terms of the Fenchel-Nielsen coordinates obtained via the length and twist parameters of curves in a pants decomposition of $S_g$.  
\begin{theorem}[\textbf{Fenchel-Nielsen}]\label{Thm:F-N}
Let $\mathcal{P}=\{\gamma_1, \gamma_2,\dots, \gamma_{3g-3}\}$ be a pants decomposition of $S_g$. Suppose $\ell_{\gamma_i}$ and $t_{\gamma_i}$ denotes the length and twist parameters of $\gamma_i$. Then 
 $$\Teich(S_g) \cong \mathbb{R}_+^{3g-3} \times \mathbb{R}^{3g-3} \cong \mathbb{R}^{6g-6}$$
  via the map 
  $$\chi=[\xi] \longmapsto (\ell_{\gamma_1}(\xi), \dots, \ell_{\gamma_{3g-3}}(\xi), t_{\gamma_1}(\xi), \dots, t_{\gamma_{3g-3}}(\xi)).$$
\end{theorem}

    \begin{figure}[h]
    \centering
    \includegraphics[width=0.8\textwidth]{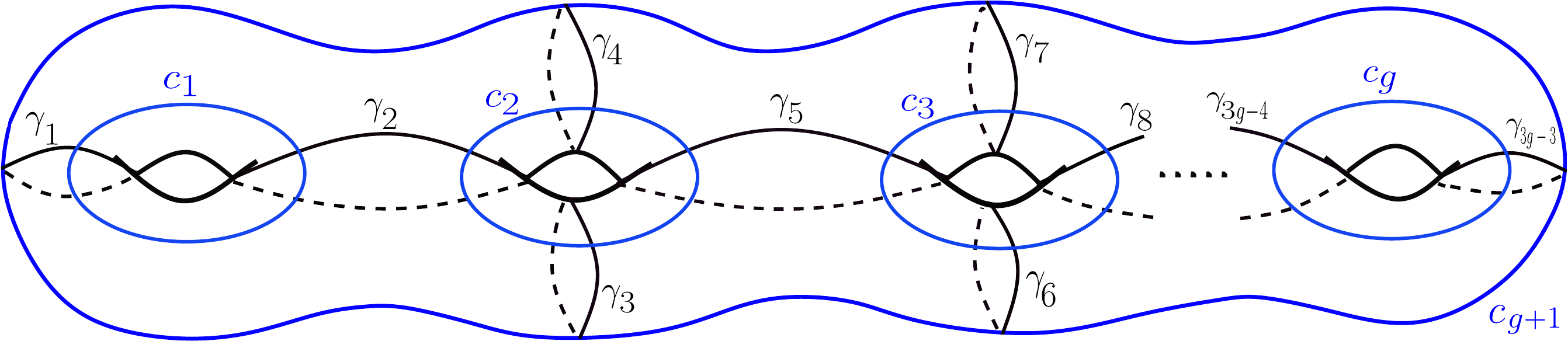}
    \caption{Pants and seams curves on $S_g$.}
    \label{fig. pants and seams curves}
\end{figure}
    
\subsection{Finite order mapping classes of $S_g$ as chain maps} \label{sub-sec:Chain}
\begin{defn}\label{chain}
Let $\Gamma=\{\alpha_1, \alpha_2, \dots, \alpha_k\}$ be a collection of essential simple closed curves on $S_g$. Then $\Gamma$ is said to be
\begin{enumerate}[(i)]
\item \textit{filling} if the surface obtained from $S_g$ after cutting along all the curves in $\Gamma$ is a disjoint union of topological 2-disks.
\item a collection of \textit{chain curves} if $i(\alpha_i, \alpha_{i+1})=1$ for all $i$ and $i(\alpha_i, \alpha_j)=0$ for all $|i-j|>1$.
\item a \textit{filling chain collection} if it is a collection of chain curves which is filling.
\end{enumerate} 
\end{defn}
\begin{figure}[h]
\centering
\includegraphics[width=0.7\textwidth]{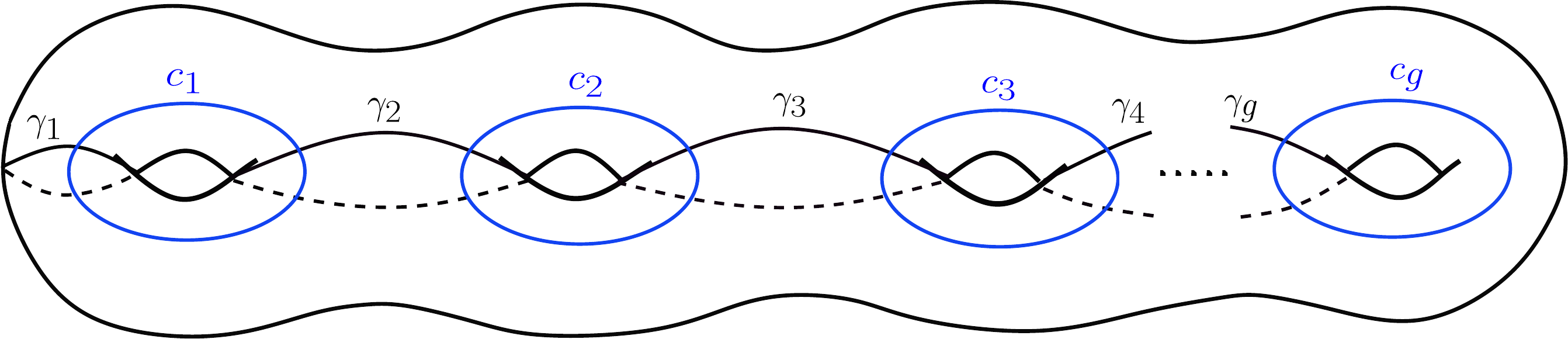}
\caption{A filling chain collection on $S_g$.}\label{Fig. 12}
\end{figure}

Let $\Gamma=\{\gamma_1, c_1, \dots, \gamma_g, c_g\}$ be the collection of the filling chain curves in $S_g$ ($g\geq 2$) as shown in Figure \ref{Fig. 12} and let $F:= T_{\gamma_1}T_{c_1} \cdot \cdot \cdot T_{\gamma_g}T_{c_g}$. Such maps are called chain maps. Then $F$ is unique upto conjugacy and it turns out that it is the irreducible Type 1 action of order $4g+2$ which represents the periodic mapping class in $\Mod(S_g)$ of the highest order. In the subsequent sections, we will focus on describing the isometric action on $\Teich(S_g)$ induced by $F$ and the Fenchel-Nielsen coordinates of its unique fixed point.
\section{Computing the induced action via compositions of Dehn twists}\label{Sec:Method1}
Following the notations introduced in Section \ref{sub-sec:twist}, $\mathcal{P}=\{\gamma_1, \gamma_2,\dots, \gamma_{3g-3}\}$, and $\mathcal{S}=\{c_1, c_2,\dots, c_{g+1}\}$ will denote the pants and the seams curves on $S_g$ ($g\geq 2$) respectively, as shown in Figure \ref{fig. pants and seams curves}. Let $\mathcal{H}=\{\gamma_1, c_1, \gamma_2, c_2, \gamma_5, c_3, \dots, \gamma_{3g-4} ,c_g, \gamma_3\}$ be the collection of $2g+1$ Humphries curves as mentioned in Theorem \ref{thm: Humphries} (see Figure \ref{fig. Humphries curves}). In this section we will compute the induced actions of Dehn twists about these Humphries curves which will lead to computing the induced actions of certain Type 1 irreducible mapping classes on $\Teich(S_g)$ as mentioned in Section \ref{sub-sec:Chain}. Note that given any diffeomorphism $f: S_g \longrightarrow S_g$ with $F=[f]\in \Mod(S_g)$ and $\chi=[\xi]\in \Teich(S_g)$ where $\xi$ is a hyperbolic metric on $S_g$, $f: (S_g, f^*(\xi)) \longrightarrow (S_g, \xi)$ becomes an isometry. We denote, $\gamma_i:=\text{length}_{\gamma_i}(\chi)=\ell_{\gamma_i}(\chi)$ and $t_i:=\text{twist}_{\gamma_i}(\chi)=t_{\gamma_i}(\chi)$. By denoting  $\Tilde{\xi}:=f^*(\xi)$, $\Tilde{\chi}:=F*\chi=[\Tilde{\xi}]$, and $\Tilde{\gamma}_i:= \text{length}_{\gamma_i}(\Tilde{\chi}) = \ell_{\gamma_i}(\Tilde{\chi})$, $\Tilde{t}_i:=\text{twist}_{\gamma_i}(\Tilde{\chi})=t_{\gamma_i}(\Tilde{\chi})$. With this notation, the action induced by $F$ on $\Teich(S_g)$ is given by 
$$F: \chi =(\gamma_1, \dots, \gamma_{3g-3}, t_1, \dots, t_{3g-3}) \longmapsto \Tilde{\chi}=(\Tilde{\gamma}_1, \dots, \Tilde{\gamma}_{3g-3}, \Tilde{t}_1, \dots, \Tilde{t}_{3g-3}),$$
where the length ($\Tilde{\gamma}_i$) and twist parameters ($\Tilde{t}_i$) can be computed by: $\Tilde{\gamma_i}=\ell_{[\gamma_i]}(\Tilde{\xi})=\ell_{\gamma_i}(f^*(\xi))=\ell_{[f(\gamma_i)]}(\xi)$, and  $\Tilde{t}_i=t_{[\gamma_i]}(\Tilde{\xi})=t_{[f(\gamma_i)]}(\xi)$. 
\subsection{Induced Action of the Dehn Twists about the Humphries generators}
 We begin by computing the induced action of Dehn twists about the pants curves in $\mathcal{H}$ and obtain the following result  
 \begin{theorem}\label{thm.4.1} 
The induced action of $T_{\gamma_i}: \Teich(S_g) \longrightarrow \Teich(S_g)$ is given by
$$(\gamma_1, \dots, \gamma_{3g-3}, t_1, \dots, t_i, \dots, t_{3g-3}) \mapsto (\gamma_1, \dots, \gamma_{3g-3}, t_1, \dots, t_i-\gamma_i, \dots, t_{3g-3}),$$
where $\gamma_i$ is a pants curve in $\mathcal{H}$.
      
\begin{proof}
Let $\chi=[\xi]=(\gamma_1, \dots, \gamma_{3g-3}, t_1, \dots, t_i, \dots, t_{3g-3})\in \Teich(S_g)$, and let $\Tilde{\chi}=T_{\gamma_i}*\chi=[\Tilde{\xi}]$, where $\Tilde{\xi}=T_{\gamma_i}^*(\xi)$ i.e. $T_{\gamma_i}: (S_g, \Tilde{\xi}) \longmapsto (S_g, \xi)$ is an isometry. Using properties of a Dehn twist, it follows that 
$$\Tilde{\gamma}_j=\ell_{[T_{\gamma_i}(\gamma_j)]}(\xi)=\gamma_j, \ \forall \ \gamma_j\in \mathcal{P} \ \text{ and } \ \Tilde{t}_i=t_{[T_{c_1}(\gamma_i)]}(\xi)=t_i-\gamma_i.$$ Similarly, $\Tilde{t}_j=t_j$ for all $j\neq i$, concluding the proof.
\end{proof}
\end{theorem}
Next,we compute the induced actions of Dehn twists about the curves in $\mathcal{H}$ that are not pants curves, but seams curves. For such curves $c_i$, two possibilities arise where $c_i$ intersects either two or four pants curves from $\mathcal{P}$. It turns out that $c_1$ and $c_g$ each intersect exactly two pants curves and all the other $c_i$'s intersect exactly four pants curves each. The first case is addressed in the following theorem.
\begin{theorem}\label{thm: Induced Tc1}
 The induced action of $T_{c_1}: \Teich(S_g) \longrightarrow \Teich(S_g)$ is given by
 $$(\gamma_1, \dots, \gamma_{3g-3}, t_1, \dots, t_{3g-3}) \mapsto (\Tilde{\gamma}_1, \Tilde{\gamma}_2, \gamma_3, \dots, \gamma_{3g-3}, \Tilde{t}_1, \Tilde{t}_2, \Tilde{t}_3, \Tilde{t}_4, t_5, \dots, t_{3g-3}),$$
 where $\Tilde{\gamma}_i$ and $\Tilde{t}_i$ are given by
\small{\begin{equation*}
\cosh{\frac{\Tilde{\gamma}_i}{2}}=|\cosh(\ln{\sqrt{r_i}})\cosh{\frac{\gamma_i-t_i-b_i}{2}} + \frac{a_i}{2\sqrt{r_i}}\sinh{\frac{\gamma_i-t_i-b_i}{2}}|, \ \text{ and }
\ \sinh{\Tilde{t}_i}=\frac{\cosh{\Tilde{\phi_i}}}{\sinh{\Tilde{L}_i}}. 
\end{equation*}}
where $a_i, b_i, r_i, \Tilde{\phi_i}$, and $\Tilde{L_i}$ are functions of known parameters described in the proof. A similar description can be provided for the induced action of $T_{c_g}$ on $\Teich(S_g)$.
\end{theorem}
\begin{figure}[h]
     \centering
        \includegraphics[width=0.72\textwidth]{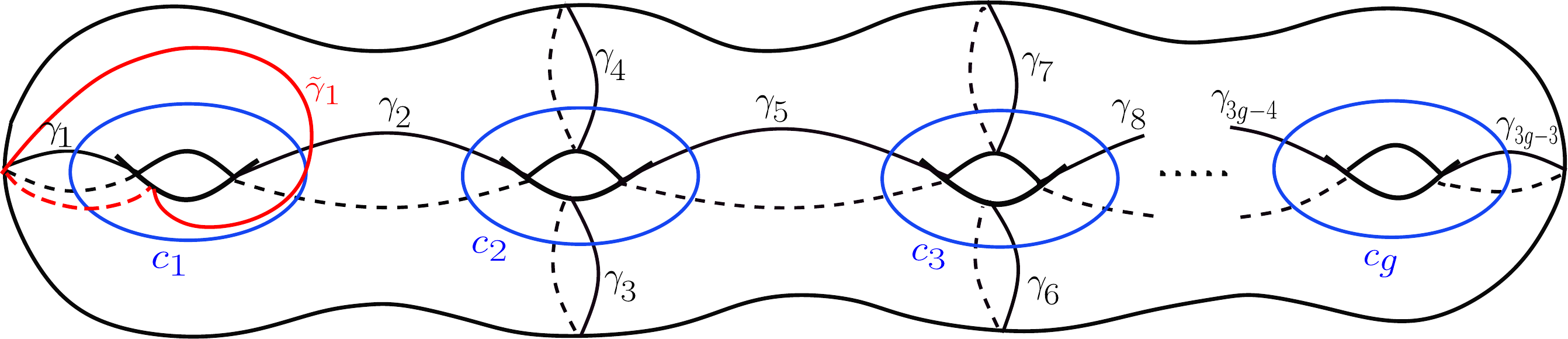}
         \caption{Action of $T_{c_1}$ on $\gamma_1$ on $(S_g, \xi)$.}
         \label{fig: action of Dehn twist on surface 1}
\end{figure}
\begin{proof} We only discuss the case of $T_{c_1}$ here. The case of $T_{c_g}$ is similar. 
   
\textbf{Length Parameters:} It follows that under the action of $T_{c_1}$, we have $\Tilde{\gamma_i}=\gamma_i$, for all $i \neq 1, 2$. To compute $\Tilde{\gamma_1}, \Tilde{\gamma_2}$, we assume that $t_1, t_2 \geq 0$, and complete the proof. For other possible values of $t_1$ and $t_2$, the proof is similar.
\begin{figure}[h]
         \centering
         \includegraphics[width=0.7\textwidth]{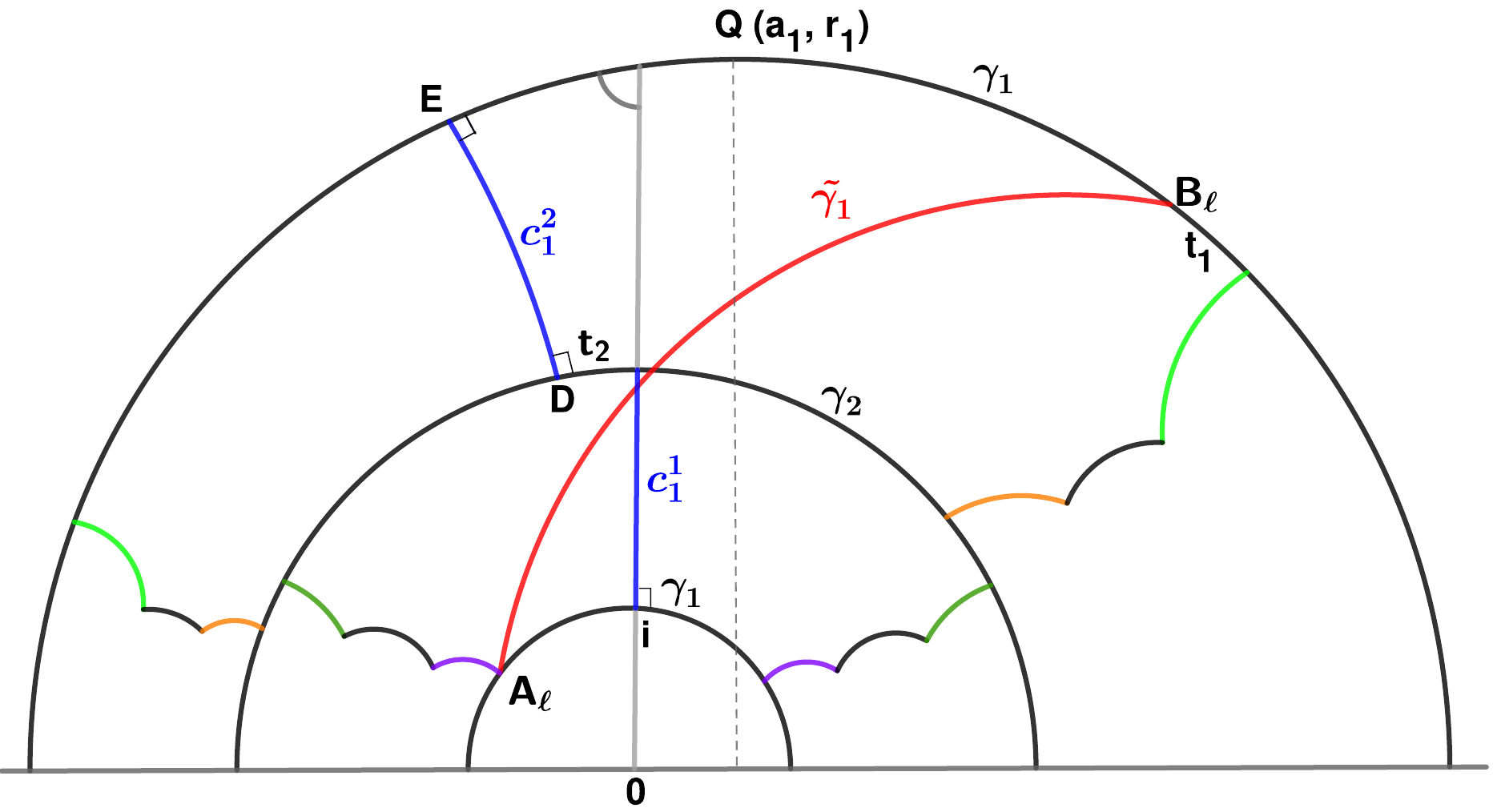}
         \caption{Lifting of $\Tilde{\gamma}_1=T_{c_1}(\gamma_1)$ onto the universal cover of $(S_g, m)$.}
         \label{fig: lifting to universal cover 1}
\end{figure}
         
Consider the universal cover (see Figure \ref{fig: lifting to universal cover 1}) of the hyperbolic surface $(S_g, \xi)$ (see Figure \ref{fig: action of Dehn twist on surface 1}). Among the two circles representing $\gamma_1$, one is the unit circle and the other is the circle centred at $(a_1, 0)$ with radius $r_1$. For $t_2 > 0$, the two common perpendiculars $c_1^1$ and $c_1^2$ between $\gamma_1$ and $\gamma_2$ along $c_1$ lie on the left of the other (following the sign convention of twist parameters). The common perpendicular $c_1^2$ is a circle when $t_2>0$ and a line segment on the imaginary axis when $t_2=0$. The geodesic $\gamma_2$ is represented by a circle centred at the origin. For $l\in [-\frac{\gamma_1}{2}, \frac{\gamma_1}{2}]$, consider the one-parameter family of geodesic segments $\Gamma_l:=[A_l, B_l]$ each of which is projected on $(S_g, \xi)$ to a simple closed curve homotopic to $ \Tilde{\gamma}_1:=T_{c_1}(\gamma_1)$. Throughout the proof the same notation will be used to denote a simple closed curve, and the length of its geodesic representative  in $(S_g, \xi)$. As $A_l$ lies on the unit circle and $l=d_{\mathbb{H}^2}(A_l, i)$, it follows that $A_l=(-\tanh{l}, \operatorname{sech}{l})$. Similarly, $B_l=(a_1+r_1\tanh{l_1}, r_1\operatorname{sech}{l_1})$ where $l_1=-l-t_1-b_1+\gamma_1$, and $b_1=d_{\mathbb{H}^2}(E, Q)$. A simple hyperbolic distance formula in $\mathbb{H}^2$ (see \cite[Theorem 1.2.6]{Katok})leads to 
\begin{eqnarray*}\label{eq.1}
 && \cosh{\Gamma_l}  = 1 + \frac{|A_l - B_l|^2}{2\operatorname{Im}(A_l)\operatorname{Im}(B_l)} \\
 && = \frac{a_1^2+r_1^2+1+2a_1\tanh{l} +2r_1\tanh{l_1}(a_1+\tanh{l})}{2r_1\operatorname{sech}{l}\operatorname{sech}{l_1}}\\
\nonumber && = \left( \frac{a_1^2+(r_1+1)^2}{4r_1}\cosh{d} + \frac{a_1}{2r_1}(1+r_1)\sinh{d} \right)\\
&& + \left( \frac{a_1^2+(r_1 - 1)^2}{4r_1}\cosh(2l-d) + \frac{a_1}{2r_1}(1-r_1)\sinh(2l-d) \right) \ (\text{putting }  \; d= \gamma_1-t_1-b_1)\\
\nonumber && = k + f(l),
\end{eqnarray*}
where $f(l) := \frac{a_1^2+(r_1 - 1)^2}{4r_1}\cosh(2l-d) + \frac{a_1}{2r_1}(1-r_1)\sinh(2l-d)$ and\\ $k := \frac{a_1^2+(r_1+1)^2}{4r_1} \cosh{d} + \frac{a_1}{2r_1}(1+r_1)\sinh{d}$ is a constant independent of $l$.

Differentiating the above equation with respect to $l$ and equating it to zero, we obtain
\begin{eqnarray}\label{eq.2}
\tanh(2l-d) &=& \frac{2a_1(r_1-1)}{a_1^2+(r_1-1)^2}= y \;(\text{say})\\
\text{ i.e. } \ l  &=& \frac{\gamma_1-t_1-b_1}{2} + \frac{1}{2}\ln{\frac{r_1+a_1-1}{r_1-a_1-1}} = l_0 \;(\text{say})
\end{eqnarray}
(as $r_1>1$ and $r_1-a_1>1$). To show that at $l=l_0$, $\cosh{\Gamma_{l}}$ i.e. $\Gamma_l$ attains its minimum value, first note that at $l=l_0$, the incidence angles of $\Gamma_l$ with $\gamma_1$ are the same, implying that $\Gamma_{l_0}$ projects to a smooth simple closed curve on $S_g$ i.e. $\Gamma_{l_0}$ is the desired geodesic segment. Let $\Gamma_{l_0}=[A_{l_0}, B_{l_0}]$ be represented by the circle of radius $r$, centred at $(a, 0)$, and let the incidence angles at $A_{l_0}$, and $B_{l_0}$ be $\theta$ and $\theta'$ respectively. Then 
\begin{equation}
\label{eq.4}
cos{\theta}=\frac{1+r^2-a^2}{2r}, \;\; \text{and}\;\; \cos{\theta'}=\frac{r^2+r_1^2-(a-a_1)^2}{2rr_1}.
\end{equation}
As $A_{l_0}$ and $B_{l_0}$ lie on the circle $(x-a)^2+y^2=r^2$, we have
\begin{equation}
\label{eq.5}
\cos{\theta} = \frac{2(a_1+\tanh{l_0}+r_1\tanh{l_1})+\tanh{l_0}(a_1^2+r_1^2-1+2a_1r_1\tanh{l_1})}{2r(a_1+\tanh{l_0}+r_1\tanh{l_1})} \ \text{and}
\end{equation} 
\begin{equation}
\label{eq.6}
\cos{\theta'} = \frac{(1+r_1^2-a_1^2)\tanh{l_1} + 2(a_1+\tanh{l_0})(r_1+a_1\tanh{l_1})}{2r(a_1+\tanh{l_0}+r_1\tanh{l_1})}. 
\end{equation}
Using equation \ref{eq.2} and $l_1=l-(2l-d))$, we have
\begin{equation*}
\tanh{l_1} = \tanh(l-(2l-d)) = \frac{\tanh{l}-y}{1-y\tanh{l}}.
\end{equation*}

\begin{equation}\label{eq.8}
    \cosh{\frac{\Tilde{\gamma}_1}{2}}=|\cosh(\ln{\sqrt{r_1}})\cosh{\frac{\gamma_1-t_1-b_1}{2}} + \frac{a_1}{2\sqrt{r_1}}\sinh{\frac{\gamma_1-t_1-b_1}{2}}|.
   \end{equation}
   
To determine $a_1, r_1$, and $b_1$, we note that $b_1=d_{\mathbb{H}^2}(E, Q)$, where $Q=(a_1, r_1)$. Let $E=(x_1, y_1)$. 
On the other hand, lengths of $c_1^1$ and $c_1^2$ (see Figure \ref{fig: lifting to universal cover 1}), are as follows (Ref. \cite[Theorem 2.4.1]{PBuserBook})
\begin{equation}\label{formula for c_1^1 and c_1^2}
\cosh{c_1^1}=\frac{\cosh{\frac{\gamma_3}{2}} + \cosh{\frac{\gamma_1}{2}}\cosh{\frac{\gamma_2}{2}}}{\sinh{\frac{\gamma_1}{2}}\sinh{\frac{\gamma_2}{2}}}, \hspace{1 cm} \cosh{c_1^2}=\frac{\cosh{\frac{\gamma_4}{2}} + \cosh{\frac{\gamma_1}{2}}\cosh{\frac{\gamma_2}{2}}}{\sinh{\frac{\gamma_1}{2}}\sinh{\frac{\gamma_2}{2}}}.
\end{equation}
Assuming $t_2 > 0$, the circle $(x-a_2)^2 + y^2 = r_2^2$ representing the common perpendicular $c_1^2$, is given by $a_2 = -e^{c_1^1}\coth{t_2}$, and $r_2 = e^{c_1^1}|\operatorname{csch}{t_2}|$. It intersects the circle $x^2+y^2=e^{2c_1^1}$, representing $\gamma_2$, perpendicularly at $D=(-e^{c_1^1}\coth{t_2}, e^{c_1^1}\operatorname{sech}{t_2})$. Then the hyperbolic M\"obius transformation along $\gamma_2$ with translation length $t_2$ given by
\begin{equation}
\label{eq.10}
T(z) = \frac{z\cosh{\frac{t_2}{2}} + e^{c_1^1}\sinh{\frac{t_2}{2}}}{ze^{-c_1^1}\sinh{\frac{t_2}{2}} + \cosh{\frac{t_2}{2}}}
\end{equation}
maps $c_1^2$ onto the imaginary axis by sending the triple $(a_2-r_2, D, a_2+r_2)$ to $(\infty, ie^{c_1^1}, 0)$ i.e. $E=(x_1, y_1)=T^{-1}(ie^{c_1^1 + c_1^2})$, giving 
\begin{equation}
\label{eq.11}
x_1=-\frac{e^{c_1^1}\sinh{t_2}}{\cosh{t_2}-\tanh{c_1^2}}, \hspace{2cm} y_1=\frac{e^{c_1^1}\operatorname{sech}{c_1^2}}{\cosh{t_2}-\tanh{c_1^2}}.
\end{equation}
Since the circle $(x-a_1)^2 + y^2 = r_1^2$ intersects $c_1^2$ and passes through $E$, it can be determined as follows 
\begin{equation}\label{eq.12}
a_1=\frac{e^{c_1^1}\sinh{t_2}}{\coth{c_1^2}-\cosh{t_2}}, \hspace{2cm} r_1=\frac{e^{c_1^1}\operatorname{csch}{c_1^2}}{\coth{c_1^2}-\cosh{t_2}}.
\end{equation}   
where it is easy to see that both $a_1$ and $r_1$ are well defined.

It follows that for $t_2>0$ (using \cite[Theorem 1.2.6 (iii)]{Katok}),  $b_1=2\operatorname{arsinh}{\sqrt{\frac{1 + \sinh^2{t_2}\sinh(2c_1^2)}{2}}}$, or equivalently, $b_1=\operatorname{arcosh}(2 + \sinh^2{t_2}\sinh(2c_1^2))$.

For $t_2=0$, $c_1^2$ is a line segment on the imaginary axis (Figure \ref{fig: lifting to universal cover 1}), and we have $D=\iota e^{c_1^1}$, $a_1=0, r_1=e^{c_1^1 + c_1^2}, E=Q, \text{ and hence } b_1=0$ in that case. This completes the proof for $t_1,t_2 \geq 0$. 
\begin{figure}[h]
\centering
\includegraphics[width=0.7\textwidth]{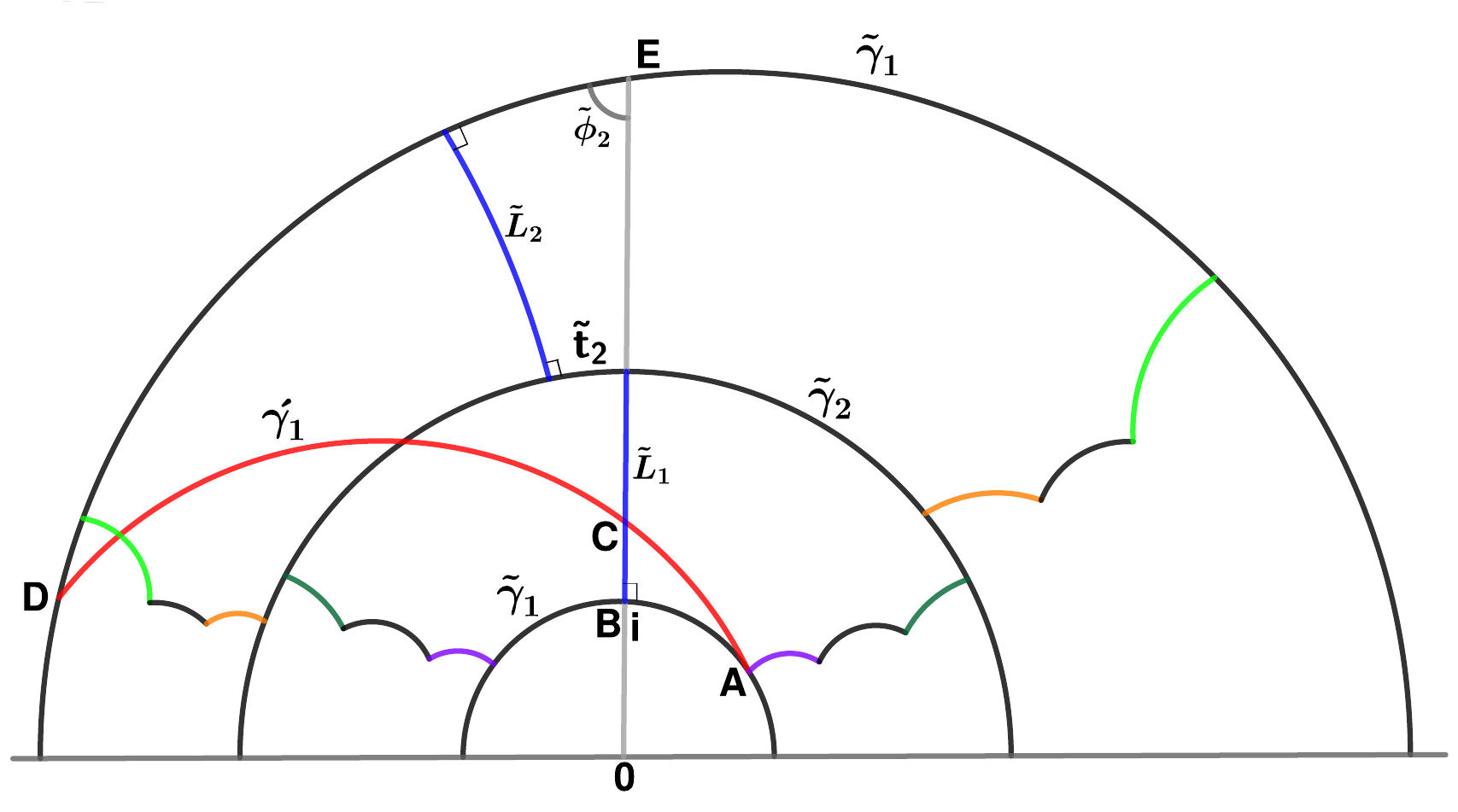}
\caption{The universal cover of $(S_g, \Tilde{m})$.}
\label{Fig. 8}
\end{figure}
         
\textbf{Twist Parameters:} As $T_{c_1}: (S_g, \Tilde{xi}) \longmapsto (S_g, xi)$ is an isometry, $\Tilde{t_i}=t_i$ for all $i \neq 1, 2, 3$. It suffices to compute $\tilde{t}_2$. $\tilde{t}_1$, $\tilde{t}_3$ can be computed similarly.
Setting $\gamma_1':=T_{c_1}^{-1}(\gamma_1)$, note that $\text{length}_{\mathbb{H}^2}(\gamma_1')=\text{length}_{(S_g, \tilde{\xi})}(\gamma_1')=\text{length}_{(S_g, \xi)}(\gamma_1)=\gamma_1$. In the triangle $\bigtriangleup ABC$, we have $\angle{B}=\frac{\pi}{2}$, $d_{\mathbb{H}^2}(A, B)=\frac{\Tilde{\gamma}_1}{2}$, and $\angle {A}=\angle({\Tilde{\gamma}_1, \gamma_1'})$ (Figure \ref{Fig. 8}) and $\angle({\Tilde{\gamma}_1, \gamma_1'})=\angle(\gamma_1, \Tilde{\gamma}_1)$ (see Figure \ref{fig: lifting to universal cover 1}) which can be seen using the formula (using Equation \ref{eq.4}) mentioned above. Let's call $d_1:=d_{\mathbb{H}^2}(A, C)$. Then the angle $\angle{C}$ and the side length $d_1$ of $\bigtriangleup ABC$ can be given by $\cos{C}=\sin{A}\cosh{\frac{\Tilde{\gamma}_1}{2}}$ (using Theorem 2.2.1 and Theorem 2.2.2 of \cite{PBuserBook}), and $\sinh{d_1}=\frac{\sinh{\frac{\Tilde{\gamma}_1}{2}}}{\sin{C}}$ (see \cite[Theorem 2.2.1]{PBuserBook}). Consider the triangle $\bigtriangleup CDE$. We have $\angle D=\angle(CDE)=\angle(BAC)=\angle A$, and $d_{\mathbb{H}^2}(D, C)=\gamma_1-d_1$. Let $\angle(CED)=\Tilde{\phi}_2$. Then we have $\cos{\tilde{\phi}_2}=\sin{A}\sin{C}\cosh(\gamma_1-d_1) - \cos{A}\cos{C}$ (by Theorem 2.2.1 of \cite{PBuserBook}). Therefore we obtain $\sinh{\Tilde{t}_2}=\frac{\cos{\Tilde{\phi}_2}}{\sinh{\Tilde{L}_2}}$ (see \cite[Theorem 2.3.1]{PBuserBook}) where $\cosh{\Tilde{L}_2}=\frac{\cosh{\frac{\gamma_4}{2}} + \cosh{\frac{\Tilde{\gamma}_1}{2}}\cosh{\frac{\Tilde{\gamma}_2}{2}}}{\sinh{\frac{\Tilde{\gamma}_1}{2}}\sinh{\frac{\Tilde{\gamma}_2}{2}}}$ (cf. \cite[Theorem 2.4.1]{PBuserBook}). This completes the proof.
\end{proof}
As an immediate consequence of Theorem \ref{thm: Induced Tc1}, we have the following result.
\begin{cor}
\label{cor.1}
\begin{enumerate}[(i)]
\item If either of the two twist parameters is zero in Thereom \ref{thm: Induced Tc1}, say, $t_j=0, j\neq i \in \{1,2\}$, then the other length parameter $\Tilde{\gamma}_i$ will be given by $\cosh{\frac{\Tilde{\gamma}_i}{2}}=\cosh{\frac{c_1^1+c_1^2}{2}}\cosh{\frac{\gamma_i-t_i}{2}}$.
\item If both the initial twist parameters are zero, i.e. $t_1=t_2=0$, then the two length parameters are given by $\cosh{\frac{\Tilde{\gamma}_i}{2}}=\cosh{\frac{c_1}{2}}\cosh{\frac{\gamma_i}{2}}$ for $i=1, 2$. 
\end{enumerate}
\end{cor}

The next theorem describes the \T map induced by $T_{c_i}$ where $c_i\in \mathcal{H}\cap\mathcal{P}$ ($2\leq i \leq g-1$) intersects exactly four pants curves in $\mathcal{P}$. We only elaborate the case $i=2$ here as all the remaining cases are similar. 

\begin{theorem}\label{thm: Induced Tc2}
The induced action of $T_{c_2}: \Teich(S_g) \longrightarrow \Teich(S_g)$ is given by
 {\small{$$(\gamma_1, . . ., \gamma_{3g-3}, t_1, . . ., t_{3g-3}) \mapsto (\gamma_1, \Tilde{\gamma}_2, . . ., \Tilde{\gamma}_5, \gamma_6, . . ., \gamma_{3g-3}, \Tilde{t}_1, . . ., \Tilde{t}_5, \Tilde{t}_6, \Tilde{t}_7, t_8, . . ., t_{3g-3}),$$}}
 
where $\Tilde{\gamma}_i$ for $2\leq i \leq 5$ are given by
$$\cosh{\frac{\Tilde{\gamma}_i}{2}}=|\cosh(\ln{\sqrt{r_i}})\cosh{\frac{\gamma_i-t_i-b_i}{2}} + \frac{a_i}{2\sqrt{r_i}}\sinh{\frac{\gamma_i-t_i-b_i}{2}}|$$
and the terms $a_i, b_i$ and $r_i$ are functions of $\gamma_j$'s and $t_j$'s. The twist parameters $\Tilde{t}_1, \dots, \Tilde{t}_7$ are given by an algorithmic method. 
\end{theorem}

\begin{figure}[h]
     \centering
        \includegraphics[width=0.7\textwidth]{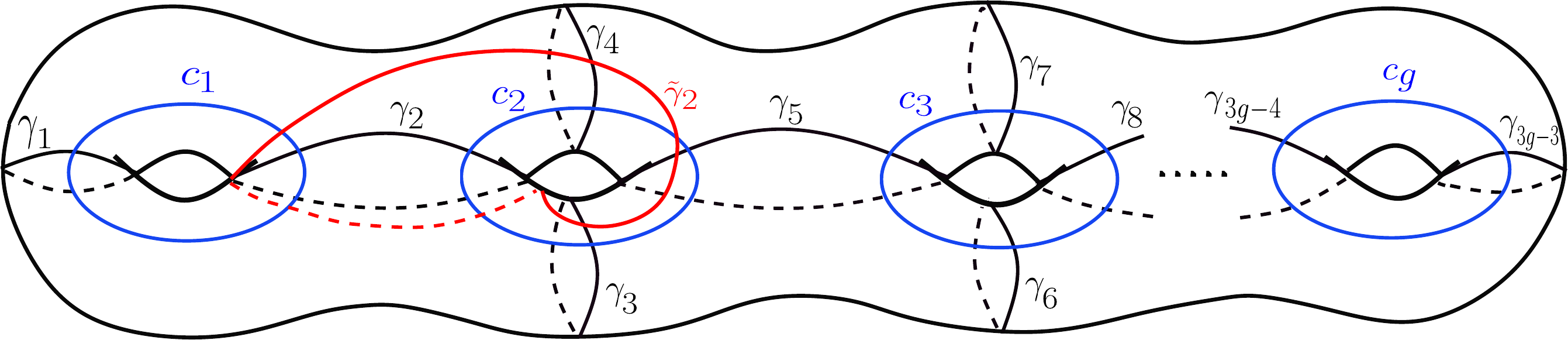}
         \caption{Action of $T_{c_2}$ on $\gamma_2$ in $(S_g, \xi)$.}
         \label{Fig. 12}
\end{figure}

\begin{proof} As the proof is very similar to that of Theorem \ref{thm: Induced Tc1}, we omit some details here to avoid repetitions.
 
\textbf{Length Parameters:}  We only compute $a_2,b_2$ and $r_2$ appearing in the formula of $\tilde{\gamma}_2$, (see Figure \ref{Fig. 11}). Similar computations follow for $a_i,b_i,r_i$ in general. As before, without loss of generality, we assume $t_1, \dots, t_5>0$.
\begin{figure}[h]
     \centering
        \includegraphics[width=0.8\textwidth]{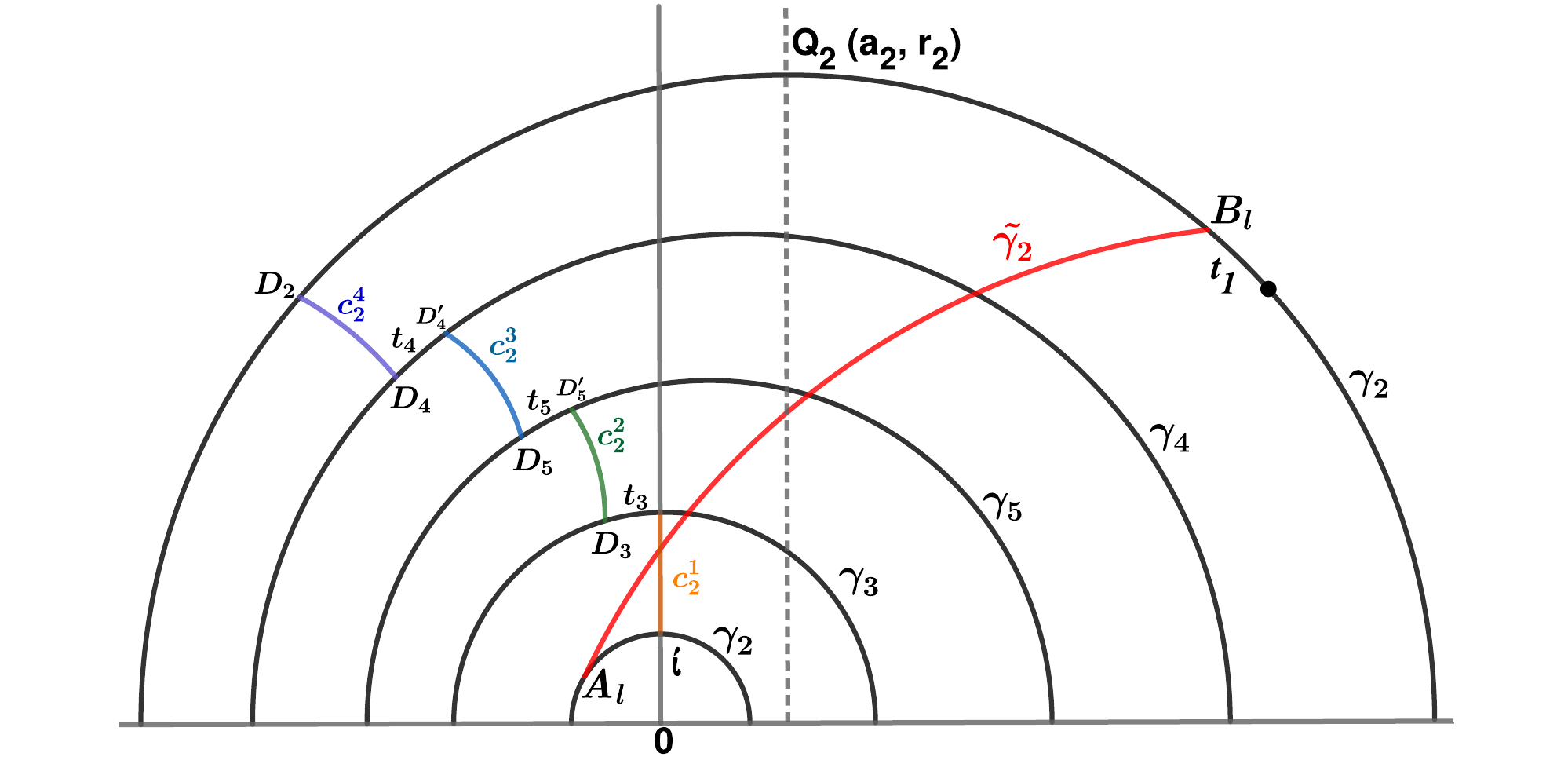}
         \caption{Lifting of $\Tilde{\gamma}_2=T_{c_2}(\gamma_2)$ onto the universal cover of $(S_g,m)$.}
         \label{Fig. 11}
\end{figure}
Let the geodesic $\gamma_3$ be represented by the circle $x^2+y^2=e^{2c_2^1}$, together with the point $D_3=(-e^{c_2^1}\tanh{t_3}, e^{c_2^1}\operatorname{sech}{t_3})$ (see Figure \ref{Fig. 11}). The common perpendicular $c_2^2$ meeting $\gamma_3$ orthogonally at $D_3$ is thus given by the circle $(x-p_2)^2+y^2=s_2^2$ with $p_2=-e^{c_2^1}\coth{t_3}$ and $s_2=e^{c_2^1}\operatorname{csch}{t_3}$. Using the M\"obius transformation $T_3(z)=\frac{z\cosh{\frac{t_3}{2}} + e^{c_2^1}\sinh{\frac{t_3}{2}}}{ze^{-c_2^1}\sinh{\frac{t_3}{2}} + \cosh{\frac{t_3}{2}}}$ sending the circle $c_2^2$ to the imaginary axis with translation length $t_3$, we obtain $D_5'=T_3^{-1}(ie^{c_2^1+c_2^2}) = (x_5', y_5')$ where $x_5'=-\frac{e^{c_2^1}\sinh{t_3}}{\cosh{t_3}-\tanh{c_2^2}}$ and $y_5'=\frac{e^{c_2^1}\operatorname{sech}{c_2^2}}{\cosh{t_3}-\tanh{c_2^2}}$. Likewise, the geodesic $\gamma_5$ is given by the circle $(x-a_5)^2+y^2=r_5^2$ where $a_5=\frac{e^{c_2^1}\sinh{t_3}}{\coth{c_2^2} - \cosh{t_3}}$ and $r_5=\frac{e^{c_2^1}\operatorname{csch}{c_2^2}}{\coth{c_2^2} - \cosh{t_3}}$. Consider the M\"obius transformations $T_4$ and $T_5$ where $T_5(z)=\frac{z\cosh{\frac{t_5}{2}} + e^{c_2^1+c_2^2}\sinh{\frac{t_5}{2}}}{ze^{-c_2^1-c_2^2}\sinh{\frac{t_5}{2}} + \cosh{\frac{t_5}{2}}}$ mapping the circle $c_2^3$ to the imaginary axis with translation length $t_5$ along $\gamma_5$, and $T_4(z)=\frac{z\cosh{\frac{t_4}{2}} + e^{c_2^1+c_2^2+c_2^3}\sinh{\frac{t_4}{2}}}{ze^{-c_2^1-c_2^2-c_2^3}\sinh{\frac{t_4}{2}} + \cosh{\frac{t_4}{2}}}$ mapping the circle $c_2^4$ onto the imaginary axis with translation length $t_4$ along $\gamma_4$. Then the points $D_2$ and $D_4$ on $\gamma_2$ and $\gamma_4$ respectively and representing the common perpendicular $c_2^4$ are $D_4=(T_4T_5T_3)^{-1}(ie^{c_2^1+c_2^2+c_2^3})$, and $D_2=(T_4T_5T_3)^{-1}(ie^{c_2^1+c_2^2+c_2^3+c_2^4})$. Consequently, the circle $(x-p_4)^2+y^2=s_4^2$ represents the geodesic segment $c_2^4$. Since $\gamma_2$ meets $c_2^4$ the orthogonally at $D_2$, the point $Q_2=(a_2, r_2)$ is determined and consequently, by the hyperbolic length formula $b_2=d_{\mathbb{H}^2}(D_2, Q_2)$, is also determined. 

\textbf{Twist Parameters:} This part is similar to the computations of twist parameters shown in Theorem \ref{thm: Induced Tc1} or steps described in Theorem \ref{thm: 10-order induced action}. In both cases, the computations are done by lifting the hyperbolic surface $\left(S_g, \tilde{\xi}\right)$ to its universal cover. In view of this, we omit those similar but not so insightful calculations. 
\end{proof}

\subsection{The induced action of a mapping class through its factorization via Dehn twists}
Given $F\in \Mod(S_g)$, its induced action $F_{\#}: \Teich(S_g) \to \Teich(S_g)$ can be explicitly computed by factorizing $F$ as a product of finitely many Dehn twists about Humphries generators and writing $F_{\#}$ as the composite of the individual actions induced by those Dehn twists. In fact, consider a mapping class $F$ written as a finite product of Dehn twists about some of the Humphries' generators (refer to the Theorem \ref{thm: Humphries}). For each of the Dehn twists coming in the presentation of $F$, we can apply Theorems \ref{thm.4.1}, \ref{thm: Induced Tc1}, and/or \ref{thm: Induced Tc2} to obtain the induced action of the Dehn twist, and finally, composing them one obtains the desired induced action $F_{\#}$ explicitly.

\begin{rem}
However, it is worth noting that the above strategy of computing the induced action of a mapping class is effective only when the explicit factorization of $F$ into Dehn twists is given. This factorization is not always available or even well understood.   Moreover, even when the factorization of a mapping class is known, composing the induced actions of the product to get $F_{\#}$ using this method becomes quite challenging and cumbersome for an arbitrary periodic $F$ as the number of factors or the genus of the surface may increase. Due to this limitation, we need to modify this method by carefully choosing the Humphries' generators for a particular $F$.   
\end{rem}

\section{The Induced Action of Finite Order Mapping Classes bypassing its factorization into Dehn twists}
In this section we discuss a modified version of Algorithm \ref{algo:1} which enables us to understand the induced action of a Finite Order Mapping Class on the \T space where the explicit factorization of the mapping class in terms of Dehn twists is not necessarily known and compute the corresponding fixed points in the \T space. The modified algorithm is stated below.
\begin{algo}\label{algo:2}  Let $F\in \Mod(S_g)$, and $\xi, \tilde{\xi} \in \HypMet(S_g)$ such that under the induced action $F_{\#}: \Teich(S_g) \longrightarrow \Teich(S_g)$, $F_{\#}([\xi])=[\tilde{\xi})] \in \Teich(S_g)$. 
\begin{enumerate}[\textit{Step} 1.] 

\item If $\mathcal{P}=\{\gamma_i\}_{i\in I}$ is the chosen pants decomposition for $S_g$, find the corresponding pants decomposition determined by $T$, that is $\tilde{\mathcal{P}}=\{T(\gamma_i)\}_{i\in I}$.
\item To compute the length of each $\tilde{\gamma_i}:=T(\gamma_i)$, cut the hyperbolic surface $(S_g, \xi)$ appropriately (possibly along the curve $\gamma_i$ if $[\gamma_i] \neq [\tilde{\gamma_i}]$) and lift it to the universal cover (the upper-half plane $\mathbb{H}^2$ or Poincar\'e disc $\mathbb{D}$).
\item In the universal cover, to find the corresponding unique geodesic representative $[\tilde{\gamma_i}]$, parametrize the family of simple closed smooth curves homotopic to $\tilde{\gamma_i}$, and minimize the length over the family of such curves.
\item For determining the twist parameters $\tilde{t_i}$ of $\tilde{\gamma_i}$, cut the surface $(S_g, \tilde{\xi})$ suitably along the curve $\tilde{\gamma_i}$ and lift it to the universal cover, and then use the method as described in Section \ref{Sec:Method1}. 
\end{enumerate}
\end{algo}
To demonstrate the algorithm we apply it on the irreducible Type 1 mapping class of order $4g+2$ on $S_g$ and study the corresponding induced action on the \T space.
\subsection{The special case of genus $g=2$} An application of Algorithm \ref{algo:2} is demonstrated in the following result.
 
\begin{theorem}\label{thm: 10-order induced action}
Let $\mathcal{P}_2=\{\gamma_1, \gamma_2, \gamma_3\}$ denote a pants decomposition of $S_2$ and $\Gamma_2=\{\gamma_1, c_1, \gamma_2, c_2 \}$ be a filling chain collection where $\mathcal{C}_2=\{c_1, c_2, c_3\}$ denotes the corresponding collection of seams curves as shown in Figure \ref{Fig. 10 order action on S2}. Then the induced action $F_{\#}: \Teich(S_2) \longrightarrow \Teich(S_2)$ of the mapping class $F=T_{\gamma_1}T_{c_1}T_{\gamma_2}T_{c_2}$ is given by 
$$(\gamma_1, \gamma_2, \gamma_3, t_1, t_2, t_3) \mapsto (c_1, c_2, c_3, \Tilde{t}_1, \Tilde{t}_2, \Tilde{t}_3)$$
 where $c_1, c_2, c_3$ are as follows
$$\cosh{\frac{c_i}{2}}=|\cosh(\ln{\sqrt{r_i}})\cosh{\frac{t_i+b_i}{2}} - \frac{a_i}{2\sqrt{r_i}}\sinh{\frac{t_i+b_i}{2}}|$$
and $\Tilde{t}_1, \Tilde{t}_2, \Tilde{t}_3$ are given by an algorithmic method.
\end{theorem} 
\begin{figure}[ht]
     \centering
     \begin{subfigure}[t]{0.4\textwidth}
         \centering
         \includegraphics[width=\textwidth]{PnS_curves_S2.pdf}
         \caption{Pants and seams curves on $S_2$.}
     \end{subfigure}
     \hspace{0.1 cm}
     \begin{subfigure}[t]{0.53\textwidth}
         \centering
         \includegraphics[width=\textwidth]{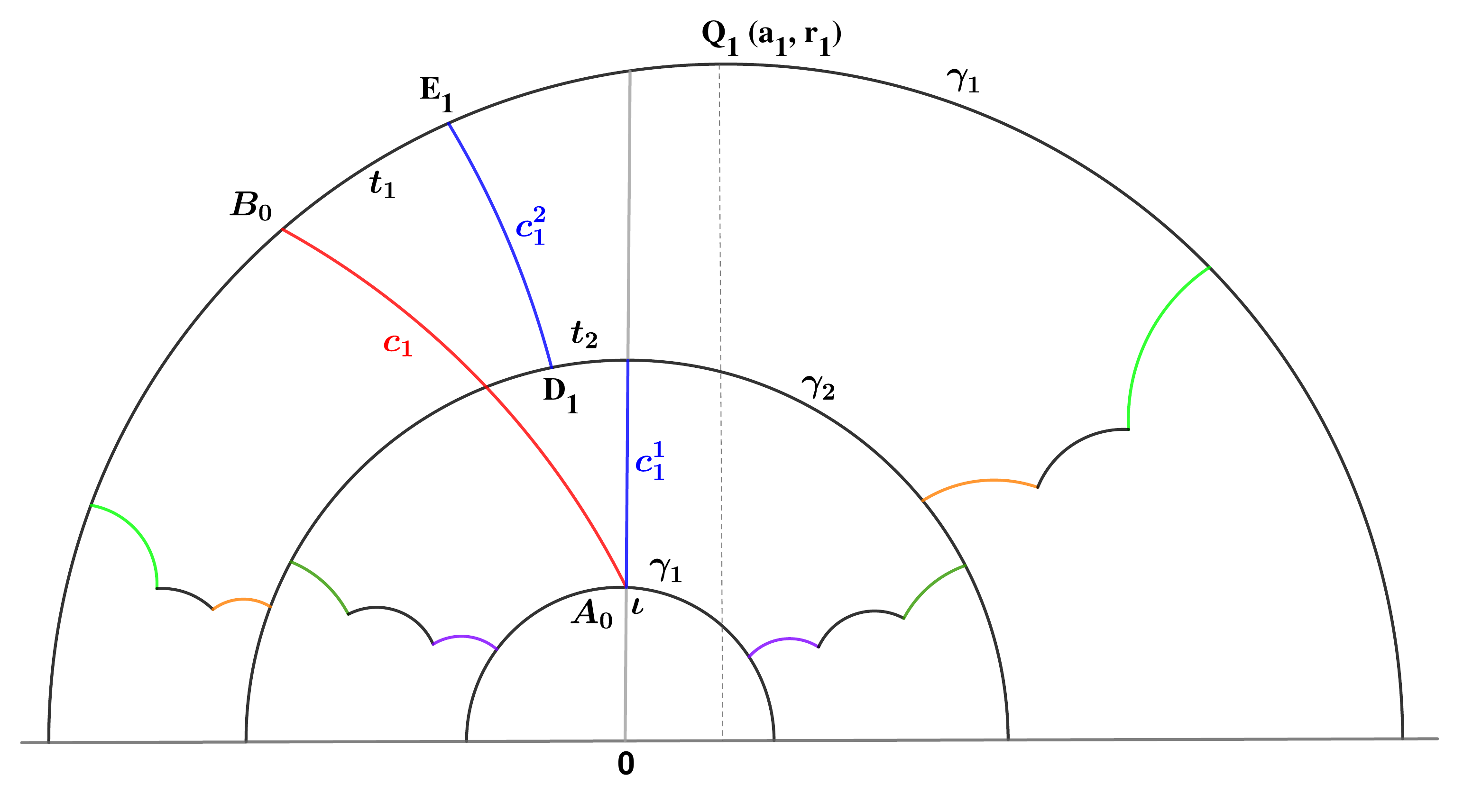}
         \caption{Universal cover of $S_2$.}
     \end{subfigure}
      \caption{10 order action on $S_2$.}
      \label{Fig. 10 order action on S2}
\end{figure} 
\begin{proof}
\textbf{Length parameters.} In this case, $[T(\gamma_i)]=[c_i]$ for $i=1, 2, 3$.  Thus $\tilde{\gamma}_i=\ell_{[\gamma_i]}(\tilde{\xi})=\ell_{[T(\gamma_i)]}(\xi)=\ell_{[c_i]}(m)=:c_i$ for any $[\xi]=(\gamma_1, \gamma_2, \gamma_3, t_1, t_2, t_3)$ in $\Teich(S_2)$. We only compute the length $\tilde{\gamma}_1$. The computations of $\tilde{\gamma}_2$ and $\tilde{\gamma}_3$ are similar.

In the upper-half plane, let the two copies of the geodesic $\gamma_1$ (see Figure \ref{Fig. 10 order action on S2}(B)) be represented by the circles $x^2 + y^2=1$ and $(x-a_1)^2 + y^2=r_1^2$. Let $B_1=(a_1, r_1)$, and $b_1=d_{\mathbb{H}^2}(B_1, D_1)$. Clearly, $t_1=d_{\mathbb{H}^2}(E_1, D_1)$. Consider the one-parameter family of line segments $\{\Gamma_l\}$ homotopic to $c_1$ where $\Gamma_l$ is the line segment joining the points $A_l=(\tanh l, \operatorname{sech}l)$ and $E_l=(a_1+r_1\tanh {l_1}, r_1\operatorname{sech}{l_1})$ where $l_1=-b_1-t_1+l$ and $l\in [-\frac{\gamma_1}{2}, \frac{\gamma_1}{2}]$ with $|l|=d_{\mathbb{H}^2}(\iota, A_l)$. Following similar argument used in the proof of Theorem \ref{thm: Induced Tc1}, we obtain
\begin{equation}
\label{eq.12}
\cosh {\Gamma_l} = 1 + \frac{|A_l-\overline{E_l}|^2}{2\text{Im}(A_l)\text{Im}(E_l)}= K+f(l),
\end{equation} 
where $K= (\frac{a_1^2+r_1^2+1}{4r_1} + \frac{1}{2})\cosh d - \frac{a_1}{2r_1}(r_1+1)\sinh d$, $f(l)=(\frac{a_1^2+r_1^2+1}{4r_1} - \frac{1}{2})\cosh (2l-d) + \frac{a_1}{2r_1}(r_1-1)\sinh (2l-d)$, and $d=b_1+t_1$. Differentiating Equation \ref{eq.12} and equating it to zero, leads to the critical point $l=l_0=\frac{d}{2} + \frac{1}{2}\ln|\frac{a_1+1-r_1}{a_1-1+r_1}|$ and we get $\tilde{\gamma}_1=c_1=\text{length}(\Gamma_{l_0})$. Thus

\begin{equation}
\label{eq:finaleqn_order10action}
\cosh {\frac{c_1}{2}} = \{\frac{1 + \cosh {c_1}}{2}\}^{\frac{1}{2}}= \frac{1}{2\sqrt{r_1}}|(r_1+1)\cosh {\frac{d}{2}} - a_1\sinh {\frac{d}{2}}|= |\cosh {\ln(\sqrt{r_1})}\cosh {\frac{d}{2}} - \frac{a_1}{2\sqrt{r_1}}\sinh {\frac{d}{2}}|.
\end{equation} 

\noindent This completes the proof as $d=b_1+t_1$. Computations to determine the values of $a_1,r_1$, and $b_1$, are similar as in the proof of Theorem \ref{thm: Induced Tc1} yielding the desired values given by Equation \ref{eq.11}.

\textbf{Twist parameters.}  To find the twist parameters, we lift the surface $(S_2, m)$ in the universal cover, and identify the hyperbolic lines that represents the curves $c_i=T(\gamma_i)$ for $i=1, 2, 3$. Now to compute $\tilde{t}_i$, find the hyperbolic lines that represents the two common perpendiculars between $c_i$ and $c_j$ for $j\in \{1, 2, 3\}\setminus \{i\}$, and their intersection points on $c_i$ in terms of Euclidean coordinates, and measure the distance in between them with the appropriate sign convetion.
\end{proof}

The following result is an important application of Theorem \ref{thm: 10-order induced action} leading to non-trivial properties of the induced map $F_{\#}: \Teich(S_2) \longrightarrow \Teich(S_2)$ when $t_1=t_2=t_3=0$.
\begin{cor}\label{cor: order-10} Given a pants decomposition $\{\gamma_1, \gamma_2, \gamma_3\}$ of $S_2$ and a filling chain collection $\{\gamma_1, c_1, \gamma_2, c_2 \}$ as shown in Figure \ref{Fig. 10 order action on S2}(A), let  $F_{\#}: \Teich(S_2) \longrightarrow \Teich(S_2)$ be the induced action of $F=T_{\gamma_1}T_{c_1}T_{\gamma_2}T_{c_2}$ as in Theorem \ref{thm: 10-order induced action}. If $t_1=t_2=t_3=0$, then $F_{\#}$ has the following properties:.
\begin{enumerate}[(i)]
\item $\Tilde{t}_1=\Tilde{t}_2=\Tilde{t}_3=0$ and 
$\cosh{\frac{c_i}{2}}=\frac{\cosh{\frac{\gamma_i}{2}}\cosh{\frac{\gamma_{i+1}}{2}} \hspace{0.1cm}+ \hspace{0.1cm}\cosh{\frac{\gamma_{i+2}}{2}}}{\sinh{\frac{\gamma_i}{2}}\sinh{\frac{\gamma_{i+1}}{2}}}$ for $i=1, 2, 3$, with $i+1, i+2$ counted up to congruence modulo $3$.
\item The sub manifold $\mathcal{H}:=\{(x_1, x_2, x_3, 0, 0, 0) | x_1, x_2, x_3>0 \} \subset \Teich(S_2)$ remains invariant under the map $F_{\#}$. In fact, $F_{\#}(\mathcal{H})=\mathcal{H}$ and $\mathcal{H}$ is a Lagrangian submanifold of $\Teich(S_2)$ with respect to its natural symplectic structure. 
\item The $1$-dimensional submanifold $\mathcal{H}_0:=\{(x, x, x, 0, 0, 0)| x>0\} \subset \mathcal{H}$ is invariant under $F_{\#}$. In fact, $F_{\#}(\mathcal{H}_0)=\mathcal{H}_0$ and $F_{\#}(x, x, x, 0, 0, 0)=(f(x), f(x), f(x), 0, 0, 0)$ where $f: \mathbb{R} \to \mathbb{R}_{+}$ is given by $f(x)=2\operatorname{arcosh}(\frac{1}{1-\operatorname{sech}{\frac{x}{2}}})$. 
\item The unique fixed point of $F_{\#}$ on $\Teich(S_2)$ is given by
 $m_0:=(x_0, x_0, x_0, 0, 0, 0)\in \mathcal{H}_0$ where $x_0=\operatorname{arcosh} 2=\ln(7+2\sqrt{3})$. 
\item $F_{\#}$ is not an Euclidean isometry.
\end{enumerate}
\end{cor} 

\begin{proof}
\begin{enumerate}[(i)] 
\item We note that $F$ interchanges the pants and seams curves of $S_2$ which means that $\mathcal{C}_2=F(\mathcal{P}_2)$ becomes a new pants decomposition of $S_2$ while $\mathcal{P}_2$ representing the corresponding seams curves. If all three initial twist parameters are zero, then it follows that all the common perpendiculars between the pants curves along the respective seams curves meet at the same points on the respective pants curves. Consequently, each pair of those common perpendiculars along each seams curve forms a simple closed geodesic homotopic to the respective seams curve (see Figure \ref{Fig. 10 order action on S2}(A)). Thus, without loss of generality, for each $j=1,2,3$, we may assume the seams curve $c_j$ to be a simple closed geodesic perpendicularly intersecting the respective unique geodesic representatives of the corresponding two pants curves. This shows that each geodesic representative $[\gamma_i]$ perpendicularly intersects $[c_j]$ whenever the corresponding curves $\gamma_i$ and $c_j$ meet. Hence the first assertion follows.

For the second part, we prove the desired result in the case of $i=1$ using Equation \ref{eq:finaleqn_order10action}. By a similar reasoning as in the Corollary \ref{cor.1}, we obtain $a_1=0, b_1=0, \text{ and } r_1=e^{c_1^1+c_1^2}$ (see Figure \ref{Fig. 10 order action on S2}(B)), giving us $c_1=c_1^1+c_1^2$. As the seams curves $[c_j]$ and the pants curves $[\gamma_i]$ intersect each other perpendicularly (provided they meet), it follows (from basic hyperbolic geometry facts) that they bisect each other. Thus we have $\frac{c_1}{2}=c_1^1=c_1^2$. The proof then follows from \cite[Theorem 2.4.1 (i)]{PBuserBook}.

\item The fact that $F_{\#}(\mathcal{H}) \subset \mathcal{H}$ follows from (i). The converse is also true as $F_{\#}$ interchanges the pants and seams curves of $S_2$. Next, $\mathcal{H}$ can be seen to be an isotropic submanifold of $\Teich(S_2)$ as using Wolpert's magic formula (see \cite[Section 10.6.5]{primer}), on $\mathcal{H}$, the standard symplectic form $d\omega=0$. Now, the second assertion follows from counting the dimension of $\mathcal{H}$.

\item From the second assertion, it follows directly that $F_{\#}(\mathcal{H}_2^1) \subset \mathcal{H}_2^1$. Due to the properties of right-angled hexagons (see \cite[Theorem 2.4.1 (i)]{PBuserBook}), the converse also holds.

\item It is straight forward to check that $f(x_2)=x_2$ where $x_2=\ln(7+2\sqrt{3})$. Now the uniqueness of the fixed point of $F_{\#}$ concludes the claim.

\item Consider $\xi:=(y, y, y, 0, 0, 0), \chi:=(z, z, z, 0, 0, 0) \in \mathcal{H}_0$ where $y=2\ln(\frac{3+\sqrt{5}}{2})$, and $z=2\ln(3 + 2\sqrt{2})$. Then it is easy to see that $\chi=F_2(\xi)$, and $\| \xi - m_2 \| \neq \| \chi - m_2 \|$, where $\|\cdot\|$ denotes the usual Euclidean norm. This completes the proof.
\end{enumerate}
\end{proof}  

\subsection{The general case: periodic mapping class of highest possible order $4g+2$ in $\mathrm{Mod}(S_g)$.}
In this section, we study the irreducible Type 1 action $F_g$ of order $4g+2$ on $S_g$ for $g \geq 3$, and understand properties of the isometric action $F_{\#}$ induced by $F_g$ on $\Teich(S_g)$ in analogy with Theorem \ref{thm: 10-order induced action} and Corollary \ref{cor: order-10}. 

We begin with a \textit{pants decomposition} $\mathcal{P}_g=\{\gamma_1, \gamma_2,\dots, \gamma_{3g-3}\}$ of $S_g$ together with the corresponding collection of \textit{seams curves} denoted by $\mathcal{C}_g=\{c_1, c_2,\dots, c_{g+1}\}$ (see Figure \ref{Fig. 4g+2 order action on Sg}). Let $\Gamma_g:=\{\gamma_1, c_1, \gamma_2, c_2, \gamma_5, c_3, \dots, \gamma_{3g-4}, c_g\}$ be a filling chain collection on $S_g$. Then $F_g= T_{\gamma_1}T_{c_1} \cdots T_{\gamma_{3g-4}}T_{c_g}$, the chain map along the filling chain $\Gamma_g$, represents the irreducible Type 1 mapping class of order $4g+2$ in $\Mod(S_g)$. For notational convenience, we denote $\beta_i := F_g(\gamma_i)$, for $i=1, \cdots, 3g-3$; $\text{Hor}(\mathcal{P}_g) := \left(\Gamma_g \cap \mathcal{P}_g\right) \cup \{\gamma_{3g-3}\} = \{\gamma_1, \gamma_2, \gamma_{2+3k}, \gamma_{3g-3} \mid 1 \leq k \leq g-2\}$, the \textit{horizontal pants curves}; and $\text{Ver}(\mathcal{P}_g) := \mathcal{P}_g \setminus \text{Hor}(\mathcal{P}_g) = \{ \gamma_{3k}, \gamma_{3k+1} \mid 1 \leq k \leq g-2\} $, the \textit{vertical pants curves}. Then we have the following theorem:

\begin{theorem}
\label{thm. 4g+2  order action on Sg}
 The induced action $F_{\#}: \Teich(S_g) \longrightarrow \Teich(S_g)$ given by  
 $$(\gamma_1, \cdots, \gamma_{3g-3}, t_1, \cdots, t_{3g-3}) \longrightarrow (\tilde{\gamma}_1, \cdots, \tilde{\gamma}_{3g-3}, \Tilde{t}_1, \cdots, \tilde{t}_{3g-3}),$$
has the following properties:
\begin{enumerate}[(i)]
\item If $t_i=0 \; \forall \; i$, then $\tilde{t}_i=0 \; \forall \; i \text{ where } \gamma_i \in \text{Hor}(\mathcal{P}_g)$.

\item
\begin{enumerate} [(a)]

\item If $t_i=0 \; \forall \; i$, and for $4 < j < 3g-6$, if $\gamma_3=\gamma_4, \gamma_6=\gamma_7, \cdots, \gamma_j=\gamma_{j+1} \in \text{Ver}(\mathcal{P}_g)$, then $\tilde{t}_i=0 \; \forall \; i \text{ such that } \gamma_i \in \mathcal{P}_g \setminus \{ \gamma_k \in \text{Ver}(\mathcal{P}_g) \mid k \geq j\}$, and $\tilde{\gamma}_3=\tilde{\gamma}_4, \tilde{\gamma}_6=\tilde{\gamma}_7, \cdots, \tilde{\gamma}_{j-3}=\tilde{\gamma}_{j-2} \in \text{Ver}(\mathcal{P}_g)$.

\item  If $t_i=0 \; \forall \; i$, and if $\gamma_3=\gamma_4, \gamma_6=\gamma_7, \cdots, \gamma_{3g-6}=\gamma_{3g-5} \in \text{Ver}(\mathcal{P}_g)$, then $\tilde{t}_i=0 \; \forall \; i$.
Moreover, we have $\tilde{\gamma}_3=\tilde{\gamma}_4, \tilde{\gamma}_6=\tilde{\gamma}_7, \cdots, \tilde{\gamma}_{3g-6}=\tilde{\gamma}_{3g-5} \in \text{Ver}(\mathcal{P}_g)$ and

$$\mathcal{H}_g^{2g-1}:=\{ (x_1, \cdots, x_{3g-3}, 0, \cdots, 0) \mid x_{3k}=x_{3k+1}, 1 \leq k \leq g-2\} \subset \Teich(S_g)$$

is an $F_{\#}$-invariant submanifold of $\Teich(S_g)$ of dimension $2g-1$.
\end{enumerate} 
\item $\mathcal{H}_g^1:=\{(x, \cdots, x, 0, \cdots, 0) | x>0\}$ is not an $F_{\#}$-invariant subspace of $\Teich(S_g)$.

\item For $g\geq 2$, the unique fixed point $m_g$ of $F_{\#}$ is given in terms of the Fenchel-Nielsen coordinates, by $m_g = (\gamma_1, \gamma_2, \dots, \gamma_{3g-3}, t_1, t_2, \dots, t_{3g-3}),$
where the twist parameters $t_i = 0$ for all $1 \le i \le 3g-3$, and the length parameters $\gamma_i$ are defined by the palindromic sequence:
\begin{equation*}
    \gamma_{3j} = \gamma_{3j+1} = \gamma_{3g-3-3j} = \gamma_{3g-2-3j} =2\operatorname{arcosh}\left(Y_j(g)\right) \quad \text{for } 1 \le j \le \left\lfloor \frac{g-1}{2} \right\rfloor,
\end{equation*}
and all the remaining length coordinates $\gamma_i$ are given by
\begin{equation*}
    \gamma_i = 2\operatorname{arcosh}\left(X(g)\right),
\end{equation*}
where $X, Y_j$ are functions of $g$ as follows
\begin{align*}
    X(g) &= 1 + 2\cos\left(\frac{\pi}{g+1}\right), \\
    Y_j(g) &= \frac{\cos\left(\frac{\pi}{g+1}\right) - \cos\left(\frac{2j+2}{g+1}\pi\right)}{1 - \cos\left(\frac{\pi}{g+1}\right)}.
\end{align*}
\end{enumerate}
\end{theorem}

\begin{proof}
\begin{enumerate}[(i)]
\item We begin by noting that under the action, $F_g(\mathcal{P}_2) = \{\beta_i\}$ is a pants decomposition of $S_2$ with $\text{Hor}(\mathcal{P}_g)$ as the seams curves, and we have, up to isotopy, $c_1=\beta_1$, $c_i=\beta_{3i-4}$ for $2 \leq i \leq g-1$, $c_g=\beta_{3g-4}$, and $c_{g+1}=\beta_{3g-3}$. Let us assume $t_i=0$ for all $i$, then all the $[c_i]$ can be assumed to be closed geodesics intersecting the pants curves $[\gamma_j]$ at right angles whenever they intersect. This follows from the property of right-angled hyperbolic hexagons that each $[c_i]$ bisects the pants curves $[\gamma_j]$ whenever they intersect. Using an argument similar to the one in Corollary \ref{cor: order-10}, the claim follows. 

\begin{figure}[h]
     \centering
        \includegraphics[width=0.85\textwidth]{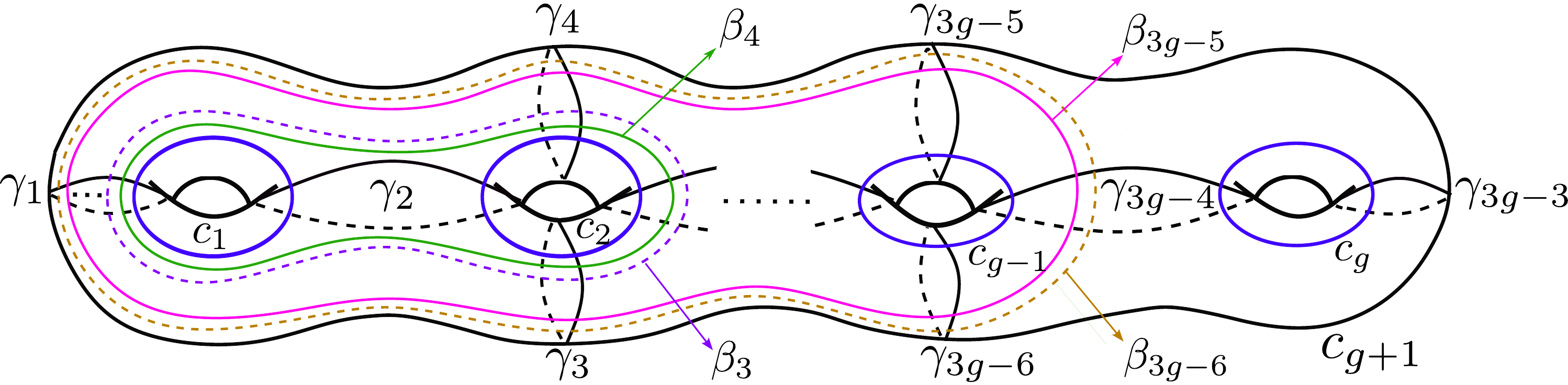}
         \caption{$4g+2$ order action on $S_g$.}
         \label{Fig. 4g+2 order action on Sg}
\end{figure}

\item From the hypothesis, it follows from (i) that $\tilde{t}_i=0$ for all $i$ such that $\gamma_i \in \text{Hor}(\mathcal{P}_g)$, and each $[c_k]$ perpendicularly bisects the pants curves $[\gamma_j]$ whenever they meet. In particular, in the $g=3$ case, we have $\gamma_3=\gamma_4$ and consequently, the curve $[c_1]$ is bisected by $[\gamma_1]$ and $[\gamma_2]$, and $[c_2]$ is bisected by $[\gamma_2]$ and $[\gamma_5]$ (using \cite[Theorem 2.4.1 (i)]{PBuserBook}). Now, consider the two common perpendiculars between $[\gamma_1]$, and $[\gamma_5]$ along the curve $\beta_4$, that form two right-angled hexagons whose sides are subarcs of $\gamma_1$, $c_1$, $\gamma_2$, $c_2$, $\gamma_5$, and $\beta_4$. As both the hexagons share three consecutive sides of lengths $\frac{c_1}{2}$, $\frac{\gamma_2}{2}$, and $\frac{c_2}{2}$, they are congruent (using \cite[Theorem 2.4.1]{PBuserBook}), and hence the two common perpendiculars along $\beta_4$, of the same length, form a closed geodesic intersecting $[\gamma_1]$ and $[\gamma_5]$ perpendicularly. Thus, we can assume $\beta_4$ to be the closed geodesic $[\beta_4]$ perpendicularly intersecting $[\gamma_1]$ and $[\gamma_5]$. As a seams curve $[\gamma_1]$ perpendicularly intersects $[c_1]$, $[\beta_4]$, and $[c_4]$ (similarly, the seams curve $[\gamma_5]$ perpendicularly intersects $[c_2]$, $[\beta_4]$, and $[c_3]$), proving that $\tilde{t}_4=0$. Likewise, we can also show that $\tilde{t}_3=0$. Using \cite[Theorem 2.4.1 (i)]{PBuserBook}, it follows that the lengths satisfy the following relation 
$$\tilde{\gamma}_3=\tilde{\gamma}_4=2\operatorname{arcosh}\left( \sinh{\frac{c_1}{2}}\sinh{\frac{c_2}{2}}\cosh{\frac{\gamma_2}{2}} - \cosh{\frac{c_1}{2}}\cosh{\frac{c_2}{2}} \right).$$

We first consider the $g=6$ case. In this case, we have $\gamma_3=\gamma_4$, and $\gamma_6=\gamma_7$. Consequently, $[c_1]$ is bisected by $[\gamma_1]$ and $[\gamma_2]$, and $[c_2]$ is bisected by $[\gamma_2]$ and $[\gamma_5]$. An argument similar to the case of $g=3$, proves that $\tilde{t}_3=\tilde{t}_4=0$, and $\tilde{\gamma}_3=\tilde{\gamma}_4$.

Assuming $\gamma_3=\gamma_4$, $\gamma_6=\gamma_7$, and $\gamma_9=\gamma_{10}$, it suffices to show that $\tilde{t}_6=\tilde{t}_7=0$, and $\tilde{\gamma}_6=\tilde{\gamma}_7$. Like $[c_1]$, and $[c_2]$, the curve $[c_3]$ also gets bisected (by $\gamma_5$ and $\gamma_8$). As before, consider the two common perpendiculars between $[\gamma_1]$, and $[\gamma_8]$ along the curve $\beta_7$, that form two congruent right-angled hexagons whose sides are sub-arcs of $\gamma_1$, $\beta_4$, $\gamma_5$, $c_3$, $\gamma_8$, and $\beta_7$, giving us $\tilde{t}_6=\tilde{t}_7=0$, and $\tilde{\gamma}_6=\tilde{\gamma}_7$. The assertion in the second part follows using similar methods as above. 

\item Given $x, y, z >0$, define $f(x,y,z):=2\operatorname{arcosh}\left( \frac{\cosh{\frac{x}{2}}\cosh{\frac{y}{2}} + \cosh{\frac{z}{2}}}{\sinh{\frac{x}{2}}\sinh{\frac{y}{2}}} \right)$, and call $m_x:=(x, \cdots, x, 0, \cdots, 0) \in \Teich(S_g)$. Under the action of $F_{\#}$ on $m_x$, we obtain $\tilde{\gamma_1}=c_1=f(x, x, x)$, and $\tilde{\gamma_2}=c_2=2f(x, x, x)$. Thus, $c_1=c_2$  would imply that $f(x, x, x)=0$, which is a contradiction, establishing the claim.

\item Restricting to the case of $g\geq 3$, let $m_g = (\gamma_1, \gamma_2, \dots, \gamma_{3g-3}, t_1, t_2, \dots, t_{3g-3}) \in \Teich(S_g)$ denote the unique fixed point of $F_{\#}$. Then $F_g$ is a Riemannian isometry of the hyperbolic surface $(S_g, m_g)$. Also, $F_g$ is a chain map and its action on the chain collection $\{\left(\gamma_1, c_1, \cdots, \gamma_{3g-4}, c_g\right)\}$ sends each curve to the next one, except the final curve $c_g$, which gets mapped to the curve $\gamma_{3g-3}$. Therefore, all the horizontal pants curves lie in a single orbit of the action, and they must have the same length, say, $x$. Next, consider the isometry $F_g^{2g+1}$ of $(S_g, m_g)$, an involution acting as a $\pi$-rotation of the surface about the axis passing through all genera of the surface. Then, each pair of vertical pants curves $\{(\gamma_{3j}, \gamma_{3j+1})\}$ (as in the hypothesis) around each genus belongs to the same orbit, and thus have the same length, say, $y_j=\gamma_{3j} = \gamma_{3j+1}$, for $1 \le j \le \left\lfloor \frac{g-1}{2} \right\rfloor$. Assuming that $m_g$ has vanishing twist parameters, similar to the case $g=2$, we note that $m_g \in \mathcal{H}_g^{2g-1}$. 
We would like to claim that $\gamma_{3j} = \gamma_{3j+1} = \gamma_{3g-3-3j} = \gamma_{3g-2-3j} (=y_j)$, for $1 \le j \le \left\lfloor \frac{g-1}{2} \right\rfloor$. To prove it for $j=1$, we consider the two isometric right-angled hexagons (one is bounded by $\gamma_1, c_1, \gamma_2, c_2, \gamma_3, c_{g+1}$, and the other one is by $\gamma_{3g-3}, c_g, \gamma_{3g-4}, c_{g-1}, \gamma_{3g-6}, c_{g+1}$) on $S_g$ (see Figure \ref{Fig. 4g+2 order action on Sg}), having three consecutive sides of  equal lengths given by $\frac{\gamma_1}{2}=\frac{c_1}{2}=\frac{\gamma_2}{2}=\frac{x}{2}=\frac{\gamma_{3g-4}}{2}=\frac{c_g}{2}=\frac{\gamma_{3g-3}}{2}$. Therefore, we get $y_1=\gamma_3=\gamma_{3g-6}=g(x, x, x)$ (using \cite[Theorem 2.4.1]{PBuserBook})), where $g(a, b, c)=2\operatorname{arcosh}\left( \sinh{\frac{a}{2}}\sinh{\frac{b}{2}}\cosh{\frac{c}{2}} - \cosh{\frac{a}{2}}\cosh{\frac{b}{2}} \right)$, for $a, b, c>0$. Similarly, we can prove the claim for $j=2$ by noting that the congruence of the aforementioned right-angled hexagons forces the respective adjacent two right-angled hexagons to also be congruent. Continuing like this, the claim follows for all $j$ thereby establishing
\begin{align}
y_j &= \gamma_{3j}=\gamma_{3j+1},
& 1\le j\le \left\lfloor \frac{g-1}{2}\right\rfloor,
\label{eq: y_j}\\
x &= \gamma_i,
& \text{for all other }
i\in\{1,2,\dots,3g-3\}.
\label{eq: x}
\end{align}
With respect to $m_g$, under the isometric action of $F_g$, all the chain curves have the same length equal to $x$. Thus, $x=\gamma_1=c_1, x=\gamma_2=c_2, \cdots, x=\gamma_{3g-4}=c_g$. By using two adjacent right-angled hexagons with a common side length given by $\frac{\gamma_{3j}}{2}$, we can easily determine the lengths of the corresponding chain curve $c_{j+1}$. Therefore, using \cite[Theorem 2.4.1]{PBuserBook}, we have the following two cases:
\begin{itemize}
\item Case (i): $g=2k$. $m_g$ is given by 
$$m_{2k}=(x, x, y_1, y_1, x, \cdots, y_{k-1}, y_{k-1}, x, y_{k-1}, y_{k-1}, x, \cdots, y_1, y_1, x, x, 0, \cdots, 0),$$
where $x, y_j$'s satisfy the following system of equations:
\begin{equation}\label{eq:even1}
\begin{split}
x &= f(x, x, y_1) \Leftrightarrow y_1=g(x, x, x),\\
x &= f(y_1, x, x) + f(y_1, x, y_2),\\
x &= f(y_2, x, y_1) + f(y_2, x, y_3),\\
x &= f(y_3, x, y_2) + f(y_3, x, y_4),\\
\vdots \\
x &= f(y_{k-2}, x, y_{k-3}) + f(y_{k-2}, x, y_{k-1}),\\
x &= f(y_{k-1}, x, y_{k-2}) + f(y_{k-1}, x, y_{k-1}),
\end{split}
\end{equation}
Setting $X:=\cosh{\frac{x}{2}}$, $Y_i:=\cosh{\frac{y_i}{2}}$ and simplifying the above equations, the following system of equations is obtained:
\begin{equation}\label{eq:even2}
\begin{split}
Y_1 &= X^3 - X^2 - X,\\
Y_2 &= X^5 - 3X^4 + 4X^2 - 1,\\
Y_j &= (X^2 - X)Y_{j-1} + XY_{j-2} - \sqrt{(X^2-1)(Y_{j-1}^2 + Y_{j-2}^2 + X^2 + 2XY_{j-1}Y_{j-2} - 1)}, 
\end{split}
\end{equation}
for $3\leq j \leq k$ with $Y_k:=Y_{k-1}$.
To linearise the system of equations \ref{eq:even2}, we define an auxiliary sequence $R_n$  as follows:
\begin{equation}
\label{eq. recurrence relation}
R_{-1} = -1, \ R_0 = 1, \ R_{n+1} = (X - 1)R_n - R_{n-1}.
\end{equation}
Consequently, we have $R_1 = X, \ R_2 = X^2 - X - 1, \ \text{ and } \ R_3 = X^3 - 2X^2 - X + 1$, which yields $R_1 R_2 = X^3 - X^2 - X = Y_1$ and $R_2 R_3 = X^5 - 3X^4 + 4X^2 - 1 = Y_2$. More generally, we claim the following coordinate transformation:
\begin{equation}
\label{eq. transformation Y_j=R_jR_{j+1}}
Y_j = R_j R_{j+1}.
\end{equation}
To establish this claim, we first make an observation (that follows from induction on $n$,) that for all $n \ge 0$, the sequence $R_n$ satisfies:
\begin{equation}
\label{eq. invariant X + 1}
R_n^2 + R_{n-1}^2 - (X - 1)R_n R_{n-1} = X + 1.
\end{equation}

It can be seen that while the non-linear recurrence for $Y_j$ involves the negative branch of a quadratic equation, assuming $Y_{j-1} = R_{j-1}R_j$ and $Y_{j-2} = R_{j-2}R_{j-1}$, the expression under the square root of the geometric formula can be simplified. Using the linear recurrence $(R_j + R_{j-2})^2 = (X-1)^2 R_{j-1}^2$ and a rearrangement of the invariant $R_j R_{j-2} = R_{j-1}^2 - (X+1)$, the expression reduces to
\begin{equation}
(X^2-1)^2 (R_{j-1}^2 - 1)^2.
\end{equation}
Taking the admissible branch where $Y_i > 1$, the right-hand side of the $Y_j$ recurrence relation simplifies to
\begin{equation}\label{eq:Y_j_simple}
Y_j = X(X-2)R_{j-1}R_j + (1-X)R_{j-1}^2 + (X^2-1).
\end{equation}
Multiplying Equation \ref{eq. invariant X + 1} at step $j$ by $(X-1)$ gives:
$$X^2 - 1 = (X-1)R_j^2 + (X-1)R_{j-1}^2 - (X-1)^2 R_j R_{j-1}.$$
Substituting $(X^2-1)$ by this expression into Equation \eqref{eq:Y_j_simple}, gives
\begin{align*}
Y_j &= X(X-2)R_{j-1}R_j + (1-X)R_{j-1}^2 \\
&\quad + \Big[ (X-1)R_j^2 + (X-1)R_{j-1}^2 - (X-1)^2 R_j R_{j-1} \Big].\\
 &= (X-1)R_j^2 - R_{j-1}R_j,\\ 
 &= R_j \Big[ (X-1)R_j - R_{j-1} \Big].
\end{align*}
By the defining recurrence relation, $(X-1)R_j - R_{j-1} = R_{j+1}$ and thus, $Y_j = R_j R_{j+1}$, proving the above claim.

The final equation $Y_k = Y_{k-1}$ representing the topological closing condition of the surface, translates to
\begin{equation}
R_k R_{k+1} = R_{k-1} R_k.
\end{equation}
Since $Y_k > 1$, $R_k \neq 0$, implying $R_{k+1} = R_{k-1}$. Substituting this into the recurrence $R_{k+1} = (X-1)R_k - R_{k-1}$ we have
$$2R_{k-1} = (X-1)R_k \implies R_k = \frac{2}{X-1}R_{k-1}.$$
Assuming $X \ge 3$, it follows that $\frac{2}{X-1} \le 1$, forcing $R_k \le R_{k-1}$. However, the sequence $R_n$ is strictly increasing in this range: in fact, $R_1 = X \ge 3 > 1 = R_0$. Assume $R_n > R_{n-1} > 0$. As $X-1 \ge 2$, we have:
$$R_{n+1} = (X-1)R_n - R_{n-1} \ge 2R_n - R_{n-1} > R_n.$$
By induction, $1 = R_0 < R_1 < R_2 < \dots < R_n$, which contradicts $R_k \le R_{k-1}$. This implies that the admissible geometric solution exists strictly for $1 < X < 3$. 

Parametrizing this range by setting $X = 1 + 2\cos\theta$ for $0 < \theta < \frac{\pi}{2}$, the characteristic equation of the linear recurrence (Equation \ref{eq. recurrence relation}) reduces to
$$\lambda^2 - (X-1)\lambda + 1 = 0 \implies \lambda^2 - 2\cos\theta\,\lambda + 1 = 0$$
with roots $\lambda = e^{\pm i\theta}$. The general solution is $R_n = A e^{in\theta} + B e^{-in\theta}$. Applying the initial conditions $R_0 = 1$ and $R_{-1} = -1$ yields:
$$A = \frac{e^{i\theta/2}}{2i\sin{\frac{\theta}{2}}}, \quad B = -\frac{e^{-i\theta/2}}{2i\sin{\frac{\theta}{2}}}.$$
Substituting these back into the general solution gives the explicit trigonometric formula:
\begin{equation}
\label{eq. formula of R_n in terms of theta}
R_n = \frac{e^{i(n+1/2)\theta} - e^{-i(n+1/2)\theta}}{2i\sin{\frac{\theta}{2}}} = \frac{\sin{(2n+1)\frac{\theta}{2}}}{\sin{\frac{\theta}{2}}}.
\end{equation}
To express $X$ and $Y_j$ explicitly in terms of the surface genus $g$, we resolve the fundamental angle $\theta$ from the geometric closing condition. At the final stage of this geometric configuration, the topological closure leads to the relation $R_{k+1} = R_{k-1}$. Substituting our explicit trigonometric formula for $R_n$ into this closing condition yields
\begin{equation*}
\frac{\sin{(2k+3)\frac{\theta}{2}}}{\sin{\frac{\theta}{2}}} = \frac{\sin{(2k-1)\frac{\theta}{2}}}{\sin{\frac{\theta}{2}}}.
\end{equation*}
Since $\sin{\frac{\theta}{2}} \neq 0$, we have
\begin{equation*}
\sin{\frac{(2k+3)\theta}{2}} = \sin{\frac{(2k-1)\theta}{2}}.
\end{equation*}
Therefore, we must have (for some integer $m$)
\begin{equation*}
\frac{(2k+3)\theta}{2} + \frac{(2k-1)\theta}{2} = \pi + 2m\pi \implies \theta = \frac{2m+1}{2k+1}\pi.
\end{equation*}
Substituting $\theta$ in Equation \ref{eq. formula of R_n in terms of theta}, as $R_n>0$ for $n\geq 1$, and $\sin{\frac{\theta}{2}}>0$ on $0<\theta<\frac{\pi}{2}$, we get:
\begin{equation*}
R_k = \frac{\sin{(2m+1)\frac{\pi}{2}}}{\sin{\frac{\theta}{2}}} = \frac{(-1)^m}{\sin \frac{\theta}{2}}.
\end{equation*}
Therefore, $m$ is a non-negative even integer. If possible, let $m\geq 2$. Set $A_n :=(2n+1)\frac{\theta}{2}$. Then the sequence $A_n$ starts at $A_0 =\frac{\theta}{2} < \frac{\pi}{2}$, ending at $A_k=(2k+1)\frac{\theta}{2}=(2m+1)\frac{\pi}{2} \geq \frac{5\pi}{2}$ with $A_n - A_{n-1}=\theta < \frac{\pi}{2}$. This would mean that there is some $n$ for which $\pi < A_n < 2\pi$, and $R_n=\frac{\sin{A_n}}{\sin{\frac{\theta}{2}}} <0$, a contradiction. Thus, we must have $m=0$. Consequently, since $g=2k$, we get $\theta=\frac{\pi}{2k+1}=\frac{\pi}{g+1}$.
Substituting this expression for $\theta$ into the parametrization $X = 1 + 2\cos\theta$ directly yields 
\begin{equation}
\label{eq. formula of X in terms of g}
X = 1 + 2\cos\left(\frac{\pi}{g+1}\right).
\end{equation}

It remains to establish the cosine formulation for the intermediate coordinates $Y_j$. Using $Y_j = R_j R_{j+1}$ and substituting the sine formula for both terms gives:
\begin{equation}
\label{eq. formula of Y_j in terms of theta}
Y_j = \frac{\sin{(2j+1)\frac{\theta}{2}}\sin{(2j+3)\frac{\theta}{2}}}{\sin^2{\frac{\theta}{2}}}.
\end{equation}
Applying the trigonometric identity, $\sin A \sin B = \frac{1}{2}[\cos(A-B) - \cos(A+B)]$ where $A = (2j+3)\frac{\theta}{2}$ and $B = (2j+1)\frac{\theta}{2}$, in Equation \eqref{eq. formula of Y_j in terms of theta}, we have:
\begin{equation*}
Y_j = \frac{\cos\theta - \cos\big((2j+2)\theta\big)}{1 - \cos\theta}.
\end{equation*}
Finally, substituting $\theta = \frac{\pi}{g+1}$ into the above expression yields
\begin{equation}
\label{eq. formula of Y_j in terms of g}
Y_j = \frac{\cos\left(\frac{\pi}{g+1}\right) - \cos\left(\frac{2j+2}{g+1}\pi\right)}{1 - \cos\left(\frac{\pi}{g+1}\right)}.
\end{equation}
\item Case (ii): $g=2k+1$. $m_g$ given by 
$$m_{2k+1}=(x, x, y_1, y_1, x, \cdots, y_{k-1}, y_{k-1}, x, y_k, y_k, x, y_{k-1}, y_{k-1}, x, \cdots, y_1, y_1, x, x, 0, \cdots, 0),$$
where $x, y_j$'s satisfy the following system of equations:
\begin{equation}
\begin{split}
x &= f(x, x, y_1) \Leftrightarrow y_1=g(x, x, x),\\
x &= f(y_1, x, x) + f(y_1, x, y_2),\\
x &= f(y_2, x, y_1) + f(y_2, x, y_3),\\
x &= f(y_3, x, y_2) + f(y_3, x, y_4),\\
\vdots \\
x &= f(y_{k-1}, x, y_{k-2}) + f(y_{k-1}, x, y_k),\\
x &= f(y_k, x, y_{k-1}) + f(y_k, x, y_{k-1})=2f(y_k, x, y_{k-1}),
\end{split}
\end{equation}
where for $a, b, c >0$, $f$ is given by $f(a, b, c)=2\operatorname{arcosh}\left( \frac{\cosh{\frac{a}{2}}\cosh{\frac{b}{2}} + \cosh{\frac{c}{2}}}{\sinh{\frac{a}{2}}\sinh{\frac{b}{2}}} \right)$.
Setting $X:=\cosh{\frac{x}{2}}$, $Y_i:=\cosh{\frac{y_i}{2}}$ and simplifying the above equations as in the above case, we have
\begin{equation}
\begin{split}
Y_1 &= X^3 - X^2 - X,\\
Y_2 &= X^5 - 3X^4 + 4X^2 - 1,\\
Y_j &= (X^2 - X)Y_{j-1} + XY_{j-2} - \sqrt{(X^2-1)(Y_{j-1}^2 + Y_{j-2}^2 + X^2 + 2XY_{j-1}Y_{j-2} - 1)}, 
\end{split}
\end{equation}
for $3\leq j \leq k+1$ with $Y_{k+1}:=Y_{k-1}$.
\end{itemize}
\end{enumerate}
In this case also, an analogous tracking of the boundary curves and the corresponding algebraic closing condition leads to the identical fundamental angle relation, $\theta = \frac{\pi}{g+1}$ and similar formulae for $X$ and $Y_j$. 
Finally, the $g=2$ case (see Corollary \ref{cor: order-10}) cn also be recovered from this general proof making the formulae of $X$ and $Y_j$ given Equations \ref{eq. formula of X in terms of g}, and \ref{eq. formula of Y_j in terms of g} applicable for all $g\ge 2$.
\end{proof}
In particular, for $g=3,4,5,$ and $6$, Theorem \ref{thm. 4g+2  order action on Sg}, has the following consequences.
\begin{cor} \label{cor: $g=3,4,5,6$}
Consider $F \in \Mod(S_g)$ and its induced action $F_{\#}: \Teich(S_g) \longrightarrow \Teich(S_g)$ as mentioned in Theorem \ref{thm. 4g+2  order action on Sg}. Then the following conclusions hold.
\begin{enumerate}
\item For $g=3$, $F_{\#}: \Teich(S_3) \longrightarrow \Teich(S_3)$given by 
 $$(\gamma_1, \cdots, \gamma_6, t_1, \cdots, t_6) \longrightarrow (\tilde{\gamma}_1, \cdots, \tilde{\gamma}_6, \Tilde{t}_1, \cdots, \tilde{t}_6)$$
has the following properties:
\begin{enumerate}[(i)]
\item $F_{\#}:(\gamma_1, \cdots, \gamma_6, 0, \cdots, 0) \mapsto (\tilde{\gamma}_1, \cdots, \tilde{\gamma}_6, 0, 0, \Tilde{t}_3, \Tilde{t}_4, 0, 0, 0)$

\item $\mathcal{H}_3^5:=\{(x_1, x_2, y, y, x_3, x_4, 0, \cdots, 0) \in \Teich(S_3) \ | \ x_1, \ldots, x_4,y >0\} $ and  
$\mathcal{H}_3^3:=\{(x_1, x_2, y, y, x_1, x_2, 0, \cdots, 0) \in \Teich(S_3)\ | \  x_1,x_2,y >0\}$ are two $F_{\#}$-invariant submanifolds of $\Teich(S_3)$ of dimensions $5$ and $3$ respectively.
\item For $x>0$, let $h(x):=2\operatorname{arcosh}\left( 1 + 2\cosh{\frac{x}{2}} \right)$. Then \\$\mathcal{H}_3^1:=\{(x, x, h(x), h(x), x, x, 0, \cdots, 0)\}$
is a $1$-dimensional $F_{\#}$-invariant subspace of $\Teich(S_3)$.

\item The unique fixed point of $F_3$ is given by $m_3=(x, x, y, y, x, x, 0, \cdots, 0),$ where $x=2\operatorname{arcosh}\left( 1 + \sqrt{2} \right), \text{ and } y=h(x)=2\operatorname{arcosh}\left( 3 + 2 \sqrt{2} \right)$.

\item $F_{\#}:(x, \cdots, x, 0, \cdots, 0) \mapsto \left(a(x), 2a(x), b(x), b(x), a(x), 2a(x), 0, \cdots, 0 \right)$, where for $X=\cosh{\frac{x}{2}}$, the functions $a, b$ are given by $a(x)=f(x, x, x)= 2\operatorname{arcosh}\left( \frac{X}{X - 1} \right)$, and $b(x)=g(z_1, z_2, x)= 2\operatorname{arcosh}\left( \frac{X(X^2-X+1)}{2X - 1} \right)$.
Thus, the one dimensional subspace $\{(x, \cdots, x, 0, \cdots, 0)\} \subset \Teich(S_3)$ is not $F_{\#}$-invariant.
\end{enumerate}
\item For $g=4$, the unique fixed point of $F_{\#}$ is given by 
$$m_4=(x, x, y, y, x, y, y, x, x, 0, \cdots, 0),$$
where $x=2\operatorname{arcosh}\left( \frac{3+\sqrt{5}}{2} \right), \text{ and } y=2\operatorname{arcosh}\left( 4+2\sqrt{5} \right)$.
\item For $g=5$, the unique fixed point of $F_{\#}$ is given by 
$$m_5=(x, x, y_1, y_1, x, y_2, y_2, x, y_1, y_1, x, x, 0, \cdots, 0),$$
where $x=2\operatorname{arcosh}\left(1 + \sqrt{3}\right), y_1=2\operatorname{arcosh}\left(5 + 3\sqrt{3}\right)$, and $y_2=2\operatorname{arcosh}\left(7 + 4\sqrt{3}\right)$.
\item For $g=6$, the unique fixed point of $F_{\#}$ is given by 
$$m_6=(x, x, y_1, y_1, x, y_2, y_2, x, y_2, y_2, x, y_1, y_1, x, x, 0, \cdots, 0),$$
where $x=2\operatorname{arcosh}\left(1 + 2\cos{\frac{\pi}{7}}\right), y_1=2\operatorname{arcosh}\left(\frac{1}{2}(-1 + \cot^2{\frac{\pi}{14}} + \csc{\frac{\pi}{14}})\right)$, and $y_2=2\operatorname{arcosh}\left(\cos{\frac{\pi}{7}} \csc{\frac{\pi}{14}}\right)$.
\end{enumerate}
\end{cor}

\begin{proof} We only elaborate the proof for the $g=3$ case, as for all $g=3, 4,5,6$, the description of the fixed point follows directly from Equations \ref{eq. formula of X in terms of g}, and \ref{eq. formula of Y_j in terms of g}. For $g=3$, the  we observe the following. 
\begin{enumerate}[(i)]
\item This is a direct consequence of Theorem \ref{thm. 4g+2  order action on Sg} (part $(i)$).

\item The first part follows from Theorem \ref{thm. 4g+2  order action on Sg} (part $(ii)$).  For the second part, as $\mathcal{H}_3^3 \subset \mathcal{H}_3^5$, it follows that for any $m=(x_1, x_2, y, y, x_1, x_2, 0, \cdots, 0) \in \mathcal{H}_3^3$, we have $F_{\#}:m\mapsto (z_1, z_2, w, w, z_3, z_4, 0, \cdots, 0)$ for some $w, z_i>0$, with $i=1, \cdots, 4$. Therefore, using \cite[Theorem 2.4.1]{PBuserBook}, we have 
\begin{align}
z_1 &= c_1= (1/2)f(\gamma_1, \gamma_2,\gamma_3) + (1/2)f(\gamma_1, \gamma_2,\gamma_4) \notag\\
&= f(x_1, x_2, y) \label{eq. cor. for g=3, eq 1}\\
&= (1/2)f(\gamma_5, \gamma_6,\gamma_3) + (1/2)f(\gamma_5, \gamma_6,\gamma_4)= c_3=z_3. \notag
\end{align}
Likewise, we have
\begin{align}
z_2 &= c_2, \notag\\
&= (1/2)f(\gamma_2, \gamma_3,\gamma_1) + (1/2)f(\gamma_2, \gamma_4,\gamma_1) +  (1/2)f(\gamma_3, \gamma_5,\gamma_6) + (1/2)f(\gamma_4, \gamma_5,\gamma_6),\notag\\
&= f(\gamma_2, \gamma_3,\gamma_1) +  f(\gamma_3, \gamma_5,\gamma_6)\notag\\
&= f(x_2, y, x_1) + f(x_1, y, x_2)= f(\gamma_1, \gamma_3,\gamma_2) +  f(\gamma_6, \gamma_3,\gamma_5) \label{eq. cor. for g=3, eq 2}\\
&= (1/2)f(\gamma_1, \gamma_3,\gamma_2) + (1/2)f(\gamma_1, \gamma_4,\gamma_2) +  (1/2)f(\gamma_3, \gamma_6,\gamma_5) + (1/2)f(\gamma_4, \gamma_6,\gamma_5)\notag\\
&= c_4=z_4. \notag
\end{align}
which concludes the claim.
\item Let $m_x=(x, x, h(x), h(x), x, x, 0, \cdots, 0) \in \mathcal{H}_3^1$ with $h(x)=2\operatorname{arcosh}\left( 1 + 2\cosh{\frac{x}{2}} \right)$. Since $m_x \in \mathcal{H}_3^3$ also, from the part (ii) it follows that
\begin{equation*}
F_{\#}(m_x)=(z_1, z_2, w, w, z_1, z_2, 0, \cdots, 0), \text{ for some } z_1, z_2, w>0.
\end{equation*}
From Equations \ref{eq. cor. for g=3, eq 1}, and \ref{eq. cor. for g=3, eq 2}, we have $z_1=f(x, x, h(x))$, and $z_2=2f(x, h(x), x)$. Also, $w$, the length of $\beta_4=F_3(\gamma_4)$ (see Figure \ref{Fig. 4g+2 order action on Sg}), is given by $w=g(c_1, c_2, \gamma_2)=g(z_1, z_2, x)$ (see \cite[Theorem 2.4.1]{PBuserBook}).

Set $X = \cosh{\frac{x}{2}}$, $Y = \cosh{\frac{h(x)}{2}}$, $Z_i = \cosh{\frac{z_i}{2}}$ for $i=1, 2$, and $W = \cosh{\frac{w}{2}}$. Then for $Y = 1 + 2X$, the coordinate components $Z_1$ and $Z_2$ are given by
\begin{align*}
    Z_1 &= \frac{X^2 + Y}{X^2 - 1} = \frac{X^2 + 2X + 1}{X^2 - 1} = \frac{X+1}{X-1} \\
    Z_2 &= \frac{2X^2(Y+1)}{(X^2-1)(Y-1)} - 1 = \frac{4X^2(X+1)}{2X(X^2-1)} - 1 = \frac{X+1}{X-1} = Z_1
\end{align*}
Thus, $z_1=z_2=z$ (say). Call $Z = \cosh{\frac{z}{2}}$. Then we have $Z = \frac{X+1}{X-1}$, or equivalently, $ X = \frac{Z+1}{Z-1}$. Hence $w = g(z,z,x)$ takes the form
\begin{equation*}
    W= \sinh^2{\frac{z}{2}} X - Z^2= (Z^2 - 1)\left( \frac{Z+1}{Z-1} \right) - Z^2= (Z+1)^2 - Z^2= 1 + 2Z.
\end{equation*}
Therefore, we get $w=h(z)$, which completes the proof.

\item From part $(iii)$ above, we get $Z=\frac{X+1}{X-1}$. Thus, solving $z=x$, equivalently, $X=Z=\frac{X+1}{X-1}$ gives the solution $X=1 + \sqrt{2}$. Thus $x=2\operatorname{arcosh}\left(1 + \sqrt{2}\right)$, and  $y=h(x)=1+2X=3+2\sqrt{2}.$

\item Clearly, we have $F_{\#}(x, \cdots, x, 0, \cdots, 0) = (z_1, z_2, w, w, z_1, z_2, 0, \cdots, 0)$ (using second part of $(ii)$ above), where $z_1=f(x, x, x)$ (from Equation \ref{eq. cor. for g=3, eq 1}), $z_2=2f(x, x, x)$ (from Equation \ref{eq. cor. for g=3, eq 2}), and $w=g(z_1, z_2, x)$ (from part $(iii)$). For $X=\cosh{\frac{x}{2}}$, it can be easily checked that $a(x)=f(x, x, x)= 2\operatorname{arcosh}\left( \frac{X}{X - 1} \right)$, and $b(x)=g(z_1, z_2, x)= 2\operatorname{arcosh}\left( \frac{X(X^2-X+1)}{2X - 1} \right)$. As $a(x) \neq 0$, the $\{(x, \cdots, x, 0, \cdots, 0)\} \subset \Teich(S_3)$ can not be an invariant subspace of $F_3$. 
\end{enumerate}
\end{proof}

The following consequence of Theorem \ref{thm. 4g+2  order action on Sg} highlights the asymptotic behaviour of the unique fixed point $m_g$.
\begin{cor}\label{thm:asymptotic-limit}
 Consider $F_g \in \Mod(S_g)$, its induced action $F_{\#}: \Teich(S_g) \longrightarrow \Teich(S_g)$, and the unique fixed point $m_g \in \Teich(S_g)$ of $F_{\#}$ as mentioned in Theorem \ref{thm. 4g+2  order action on Sg}. As $g \to \infty$, the Fenchel-Nielsen length coordinates $X(g)$ and $Y_j(g)$ of $m_g$ converge point-wise to universal integer values that are independent of $g$ as follows
\begin{align}
    \lim_{g \to \infty} X(g) &= 3,\notag \\
    \lim_{g \to \infty} Y_j(g) &= (2j+1)(2j+3) = 4j^2 + 8j + 3.\notag
\end{align}
\end{cor}

\begin{proof}
The limiting behaviour follows from the continuity of the trigonometric parametrizations of $X(g)$ and $Y_j(g)$ since the fundamental angle $\theta_g = \frac{\pi}{g+1} \to 0$ (see Theorem \ref{thm. 4g+2  order action on Sg}, part $(iv)$ for details) as $g\to \infty$ which immediately implies that $X_g \to 0$ (using the expression of $X(g)$).

For coordinates $Y_j(g)$, the limit initially takes an indeterminate form, which we resolve via Taylor series expansions. Recall that $Y_j(g)$ in terms of $\theta_g$ is given by
\begin{equation*}
    Y_j(g) = \frac{\cos\theta_g - \cos\big((2j+2)\theta_g\big)}{1 - \cos\theta_g}.
\end{equation*}
Using the standard Taylor series expansion for the cosine function near zero, $\cos x = 1 - \frac{x^2}{2} + \mathcal{O}(x^4)$ and applying it to each term in the numerator and denominator yields
\begin{align*}
    Y_j(g) &= \frac{\left(1 - \frac{\theta_g^2}{2} + \mathcal{O}(\theta_g^4)\right) - \left(1 - \frac{(2j+2)^2\theta_g^2}{2} + \mathcal{O}(\theta_g^4)\right)}{1 - \left(1 - \frac{\theta_g^2}{2} + \mathcal{O}(\theta_g^4)\right)} \\
    &= \frac{\frac{\theta_g^2}{2}(2j+1)(2j+3) + \mathcal{O}(\theta_g^4)}{\frac{\theta_g^2}{2} + \mathcal{O}(\theta_g^4)},\\
    &= \frac{(2j+1)(2j+3) + \mathcal{O}(\theta_g^2)}{1 + \mathcal{O}(\theta_g^2)}.
\end{align*}

Taking the limit as $\theta_g \to 0^+$, the higher-order terms vanish, and we have
\begin{equation}
\label{eq. limiting value of Y_j}
    \lim_{g \to \infty} Y_j(g) = (2j+1)(2j+3)= 4j^2 + 8j + 3.
\end{equation}
This completes the proof.
\end{proof}
\begin{rem}
The asymptotic behaviour of the length parameters $X(g)$ and $Y_j(g)$ implies in particular that as $g \to \infty$, lengths of the geodesic loops remain bounded from below by a strictly positive constant. Thus, in the limiting case, the injectivity radius stays bounded from below and there is no local collapsing.
\end{rem}
\section*{Acknowledgement} For this work, the second author was partially supported by the University Grants Commission (UGC), Government of India. The authors would like to thank Dr. Pabitra Barman and Dr. K. Srinivasa Raghava for their insightful discussions and valuable suggestions during the development of this work.
\bibliographystyle{plain}
\bibliography{arxiv_version}
\end{document}